\documentclass[11pt, final]{article}
\usepackage{a4}
\usepackage[left=2cm,right=2cm,top=1.5cm,bottom=1.5cm]{geometry}
\usepackage[usenames,dvipsnames]{xcolor}
\usepackage[plainpages=false,pdfpagelabels,colorlinks=true,linkcolor=blue,citecolor=Magenta]{hyperref}

\usepackage{ucs}
\usepackage[utf8]{inputenc}
\usepackage[T1]{fontenc}
\usepackage[english]{babel}
\usepackage{amsmath,amssymb,amstext,amsthm}
\usepackage{cases}
\usepackage{graphicx}
\usepackage{mathtools}
\mathtoolsset{showonlyrefs}
\usepackage{authblk}
\usepackage{multirow}
\usepackage{booktabs}

\newtheorem{theorem}{Theorem}
\newtheorem{definition}[theorem]{Definition}
\newtheorem{lemma}[theorem]{Lemma}
\newtheorem{proposition}[theorem]{Proposition}

\newtheorem{rem}[theorem]{Remark}
\newenvironment{remark}{\rem\rm}{\endrem}
\newtheorem{assumption}{Assumption}

\newcommand{\R}{\mathbb{R}}
\newcommand{\cH}{{\mathcal H}}
\newcommand{\cG}{{\mathcal G}}
\newcommand{\abs}[1]{\left\lvert #1 \right\rvert}
\newcommand{\norm}[1]{\left\lVert #1 \right\rVert}
\newcommand{\normsq}[1]{\norm{#1}^{2}}
\newcommand{\sprod}[2]{\left\langle #1, #2 \right\rangle}
\newcommand{\prox}[3][]{\operatorname{prox}^{#1}_{#2}\left(#3 \right)}
\newcommand{\ny}{\nabla_{y}}
\newcommand{\ph}{\, \cdot \,}
\newcommand{\minus}{\scalebox{0.75}[1.0]{$-$}}

\DeclareMathOperator{\tr}{trace}
\DeclareMathOperator{\diag}{diag}
\DeclareMathOperator{\gra}{gra}
\DeclareMathOperator{\sgn}{sgn}

\DeclareMathOperator*{\argmin}{arg\,min}

\DeclareMathOperator*{\dom}{dom}
\DeclareMathOperator*{\pr}{Pr}

\DeclareOldFontCommand{\bf}{\normalfont\bfseries}{\mathbf}

\title{An accelerated minimax algorithm for convex-concave saddle point problems with nonsmooth coupling function}

\author[1,2,$\ast$]{\normalsize Radu~Ioan~Bo{\c t}}
\author[1,$\ast$]{Ern{\" o}~Robert~Csetnek}
\author[2,$\ast$]{Michael~Sedlmayer}
\affil[1]{Faculty~of~Mathematics}
\affil[2]{Research~Network~Data~Science~@~Uni~Vienna}
\affil[$\ast$]{University~of~Vienna, Austria}
\affil[ ]{\texttt{ \{radu.bot, robert.csetnek, michael.sedlmayer\}@univie.ac.at} }

\begin{document}
\maketitle

\begin{abstract}
	\noindent
	In this work we aim to solve a convex-concave saddle point problem, where the convex-concave coupling function is smooth in one variable and nonsmooth in the other and \emph{not} assumed to be linear in either. The problem is augmented by a nonsmooth regulariser in the smooth component. We propose and investigate a novel algorithm under the name of \emph{OGAProx}, consisting of an \emph{optimistic gradient ascent} step in the smooth variable coupled with a proximal step of the regulariser, and which is alternated with a \emph{proximal step} in the nonsmooth component of the coupling function.
	We consider the situations convex-concave, convex-strongly concave and strongly convex-strongly concave related to the saddle point problem under investigation. Regarding iterates we obtain (weak) convergence, a convergence rate of order $ \mathcal{O}(\frac{1}{K}) $ and linear convergence like $\mathcal{O}(\theta^{K})$ with $ \theta < 1 $, respectively. In terms of function values we obtain ergodic convergence rates of order $ \mathcal{O}(\frac{1}{K}) $, $ \mathcal{O}(\frac{1}{K^{2}}) $ and $ \mathcal{O}(\theta^{K}) $ with $ \theta < 1 $, respectively.
	We validate our theoretical considerations on a nonsmooth-linear saddle point {problem, the training of multi kernel support vector machines and a classification problem incorporating minimax group fairness.}

	\begin{keywords}
		saddle point problem, convex-concave, minimax algorithm, convergence rate, acceleration, linear convergence
	\end{keywords}
\end{abstract}

\section{Introduction}
Saddle point -- or minimax -- problems have witnessed increased interest due to many relevant and challenging applications in the field of machine learning, with the most prominent being the training of Generative Adversarial Networks (GANs)~\cite{GAN}. Even though the problems in reality are often not of this form, in the classical setting the minimax objective comprises a smooth convex-concave coupling function with Lipschitz continuous gradient and a (potentially nonsmooth) regulariser in each variable, leading to a convex-concave objective in total.

One well established method in practice due to its simplicity and computational efficiency is \emph{Gradient Descent Ascent} (GDA), either in a \emph{simultaneous} or in an \emph{alternating} variant (for a recent comparison of the convergence behaviour of the two schemes we refer to \cite{ZWLG}). However, naive application of GDA is known to lead to oscillatory behaviour or even divergence already in simple cases such as bilinear objectives. Most algorithms with convergence guarantees in the general convex-concave setting make use of the formulation of the first order optimality conditions as monotone inclusion or variational inequality, treating both components in a symmetric fashion. For example we have the \emph{Extragradient method}~\cite{EG} whose application to minimax problems has been studied in~\cite{MirrorProx} under the name of \emph{Mirror Prox}, and the \emph{Forward-Backward-Forward method} (FBF)~\cite{FBF} with application to saddle point problems in~\cite{2steps}. Both algorithms have even been successfully applied to the training of GANs (see~\cite{Gidel, 2steps}), but, though being single-loop methods, suffer in practice from requiring two gradient evaluations per iteration. A possible way to avoid this is to reuse previous gradients. Doing this for FBF -- as shown in~\cite{2steps} -- recovers the \emph{Forward-Reflected Backward method}~\cite{FRB} which was applied to saddle point problems under the name of \emph{Optimistic Mirror Descent} and to GAN training under the name of \emph{Optimistic Gradient Descent Ascent}~\cite{OGDA1, OGDA2, OGDA3}.

The first method treating general coupling functions with an asymmetric scheme is the \emph{Accelerated Primal-Dual Algorithm} (APD) by~\cite{Aybat}, involving an optimistic gradient ascent step in one component which is followed by a gradient descent step in the other one. In the special case of a bilinear coupling function APD recovers the \emph{Primal-Dual Hybrid Gradient Method} (PDHG)~\cite{PDHG}. In the case of the minimax objective being strongly convex-concave~ acceleration of PDHG is obtained  in \cite{PDHG}, which is also done for APD in~\cite{Aybat}, however only under the rather limiting assumption of linearity of the coupling function in one component.

In this paper we introduce a novel algorithm \emph{OGAProx} for solving a convex-concave saddle point problem, where the convex-concave coupling function is smooth in one variable and nonsmooth in the other, and it is augmented by a nonsmooth regulariser in the smooth component. OGAProx consists of an optimistic gradient ascent step in the smooth component of the coupling function combined with a proximal step of the regulariser, and which is followed by a proximal step of the coupling function in the nonsmooth component. We will be also able to accelerate our method in the convex-strongly concave setting \emph{without} linearity assumption on the coupling function. Furthermore we prove linear convergence if the problem is strongly convex-strongly concave, yielding similar results as for PDHG~\cite{PDHG} in the bilinear case.

So far in most works nonsmoothness is only introduced via regularisers, as the coupling function is typically accessed through gradient evaluations. Recently there is another development, although with the saddle point problem \emph{not} being convex-concave, where the assumption on differentiability of the coupling function in both components is weakened to only one component~\cite{altGDA}. As the evaluation of the proximal mapping does not require differentiability we will assume the coupling function to be smooth in only one component, too.

The remainder of the paper is organised as follows. Next we will introduce the precise problem formulation and the setting we will work with, formulate the proposed algorithm OGAProx and state our contributions. This will be followed by preliminaries in Section~\ref{sec:prelim}. Afterwards we will discuss the properties of our algorithm in the convex-concave and convex-strongly concave setting and state respective convergence results in Section~\ref{sec:convex-strconcave}. After that we will investigate the convergence of the method under the additional assumption of strong convexity-strong concavity in Section~\ref{sec:strconvex-strconcave}. The paper will be concluded by numerical experiments in Section~\ref{sec:numerics}, where we treat a simple nonsmooth-linear saddle point {problem, the training of multi kernel support vector machines and a classification problem taking into account minimax group fairness.}

%\section{Problem Description and Algorithm}\label{sec:prob-alg-contrib}
\subsection{Problem description}
Consider the saddle point problem
\begin{equation}
	\label{min-max}
	\min_{x\in \cH} \max_{y\in\cG} \Psi(x,y) := \Phi(x,y) - g(y),
\end{equation}
where $ \cH, \, \cG $ are real Hilbert spaces, $ \Phi: \cH \times \cG \to  \R \cup \{+\infty \}$ is a coupling function  with
$\dom \Phi := \{(x,y) \in  \cH \times \cG \ | \ \Phi(x,y) < +\infty\} \neq \emptyset$ and $ g: \cG \to \R \cup \{+\infty \}$ a regulariser. Throughout the paper (unless otherwise specified) we will make the following assumptions:
 \begin{enumerate}
 \item[$\bullet$] $g$ is proper, lower semicontinuous and convex with modulus  $\nu \geq 0$, i.e. $ g - \frac{\nu}{2} \normsq{\ph} $ is convex (notice that we also allow and consider the situation $\nu = 0$, in which case $g$ is convex; otherwise $g$ is strongly convex);
  \item[$\bullet$] for all $ y \in \dom g $, $ \Phi(\ph,y): \cH \to \R \cup \{+\infty \}$ is proper, convex and lower semicontinuous;
\item[$\bullet$] for all $x \in \pr_{\cH} (\dom \Phi) := \{ u \in \cH \ | \ \exists y \in \cG \ \mbox{such that} \ (u,y) \in \dom \Phi \}$ we have that  $\dom \Phi(x, \ph) = \cG$ and $\Phi(x,\ph): \cG\to \R $ is concave and Fr{\'e}chet differentiable. Moreover, $\pr_{\cH} (\dom \Phi) $ is closed;
\item[$\bullet$] there exist $ L_{yx}, \, L_{yy} \geq 0$ such that for all $ (x,y), \, (x',y') \in \pr_{\cH} (\dom \Phi) \times \dom g $ it holds
\begin{equation}
	\label{cond-Phi_y}
	\norm{\ny \Phi (x, y) - \ny \Phi (x', y')} \leq L_{yx} \norm{x-x'} + L_{yy} \norm{y - y'}.
\end{equation}
 \end{enumerate}

By convention we set $ +\infty - (+\infty) := +\infty $. Thus, the situation can be summarised by
\begin{equation}
	\label{psi}
	\Psi(x, y) =
	\begin{cases}
		- \infty & \text{if } x \in \pr_{\cH} (\dom \Phi) \text{ and } y \notin \dom g,\\
		\Phi(x,y) - g(y) & \text{if } x \in \pr_{\cH} (\dom \Phi) \text{ and } y \in \dom g,\\
		+ \infty & \text{if } x \notin \pr_{\cH} (\dom \Phi).
	\end{cases}
\end{equation}

\noindent We are interested in finding a \emph{saddle point} of~\eqref{min-max}, which is a point $ (x^{\ast}, y^{\ast}) \in \cH \times \cG $ that fulfils the inequalities
\begin{equation}
	\label{def-saddle-point}
	\Psi(x^{\ast}, y) \leq \Psi(x^{\ast}, y^{\ast})\leq \Psi(x, y^{\ast}) \quad \forall (x,y) \in \cH \times \cG.
\end{equation}
For the remainder we assume that such a saddle point exists.

The assumptions considered above ensure that for any saddle point $ (x^{\ast}, y^{\ast}) \in \cH \times \cG $ we have
$$x^{\ast}\in \pr\nolimits_{\cH} (\dom \Phi), \quad y^{\ast}\in \dom g \ \mbox{and} \ \Psi(x^{\ast}, y^{\ast}) = \Phi (x^{\ast},y^{\ast}) - g(y^{\ast})\in\R.$$

Finding a saddle point of~\eqref{min-max} amounts to solving the necessary and sufficient first order optimality conditions, given by the following coupled inclusion problems
\begin{equation}
	0 \in \partial \left[ \Phi (\ph, y^{\ast}) \right] (x^{\ast})
	\hspace{5mm}
	\text{and}
	\hspace{5mm}
	0 \in \minus \ny \Phi (x^{\ast}, y^{\ast}) + \partial g(y^{\ast}).
\end{equation}

\begin{remark} In case $\Phi$ and $g$ have full domain, $\Psi$ is a convex-concave function with full domain and the set $ \pr_{\cH} (\dom \Phi) $ is obviously closed. However, in order to allow more flexibility and to cover a wider range
of problems (see also the last section with numerical experiments), our investigations are carried out in the more general setting given by the assumptions described above. Furthermore, these assumptions allow us to stay in the rigorous setting of the theory of convex-concave saddle functions as described by Rockafellar in \cite{TR} (see Definition \ref{def-rock} and Proposition \ref{prop:saddle-function} below).
\end{remark}

\subsection{Algorithm}
The algorithm we investigate performs an optimistic gradient ascent step of $\Phi$ followed by an evaluation of the proximal mapping $g$ in the variable $y$, while it carries out a purely proximal step of $\Phi$ in $x$. We will call this method \emph{Optimistic Gradient Ascent -- Proximal Point algorithm} (OGAProx) in the following. For all $ k \geq 0 $ we define
\newsavebox{\mycases}% Store case "title" and brace
\begin{align}
	\sbox{\mycases}
	{$
		\displaystyle
		\left\{
		\begin{array}{@{}c@{}}
			\vphantom{
				y_{k+1} := \prox{\sigma_{k} g}{y_{k} + \sigma_{k} \left[ (1 + \theta_{k}) \ny \Phi(x_{k}, y_{k}) - \theta_{k} \ny \Phi(x_{k-1}, y_{k-1})\right] },
			}\\
			\vphantom{
				x_{k+1} := \prox{\tau_{k} \Phi(\ph, y_{k+1})}{x_{k}},
			}
		\end{array}
		\right.\kern-\nulldelimiterspace
	$}
	\raisebox{-.5\ht\mycases}[0pt][0pt]{\usebox{\mycases}}
	\label{alg_y}
	y_{k+1} & = \prox{\sigma_{k} g}{y_{k} + \sigma_{k} \left[ (1 + \theta_{k}) \ny \Phi(x_{k}, y_{k}) - \theta_{k} \ny \Phi(x_{k-1}, y_{k-1})\right] },\\
	\label{alg_x}
	x_{k+1} & = \prox{\tau_{k} \Phi(\ph, y_{k+1})}{x_{k}},
\end{align}
with the conventions $ x_{-1} := x_{0} $ and $ y_{-1} := y_{0} $ for starting points $ x_{0} \in \pr_{\cH} (\dom \Phi) $ and $ y_{0} \in \dom g $. The particular choices of the sequences $ (\sigma_{k})_{k \geq 0}, \, (\tau_{k})_{k \geq 0} \subseteq \R_{++} $ and $ (\theta_{k})_{k \geq 0} \subseteq \left( 0, 1 \right] $ will be specified later.

\subsection{Contribution}
Let us summarize the main results of this paper:
\begin{enumerate}
	\item We introduce a novel algorithm to solve saddle point problems with nonsmooth coupling functions, which in general is \emph{not} assumed to be linear in either component.

	\item We prove for the saddle function $ \Psi $ being
	\begin{enumerate}
		\item convex-concave (see Theorem \ref{convex-concave}):
		\begin{itemize}
			\item weak convergence of the generated sequence $ (x_{k}, y_{k})_{k \geq 0} $ to a saddle point $ (x^{\ast}, y^{\ast}) $ as $ k \to + \infty $;

			\item convergence of the minimax gap $\Psi(\bar{x}_{K},y^{\ast}) - \Psi(x^{\ast},\bar{y}_{K}) $ to zero like $ \mathcal{O}(\frac{1}{K})$ as $K \to + \infty $, where $(\bar{x}_{K})_{K \geq 1}$ and $(\bar{y}_{K})_{K \geq 1}$ are the \emph{ergodic sequences} obtained by averaging $(x_k)_{k \geq 1}$ and $(y_k)_{k \geq 1}$, respectively;
		\end{itemize}

		\item convex-strongly concave (see Theorem \ref{convex-str-concave}):
		\begin{itemize}
			\item strong convergence of $ (y_{k})_{k \geq 0} $ to $ y^{\ast} $ like $ \mathcal{O}(\frac{1}{k})$ as $ k \to + \infty $;

			\item convergence of the minimax gap  $ \Psi(\bar{x}_{K},y^{\ast}) - \Psi(x^{\ast},\bar{y}_{K}) $ to zero like $ \mathcal{O}(\frac{1}{K^{2}})$ as $K \to + \infty $;
		\end{itemize}

		\item strongly convex-strongly concave (see Theorem \ref{str-convex-str-concave}):
		\begin{itemize}
			\item linear convergence of $ (x_{k}, y_{k})_{k \geq 0} $ to $ (x^{\ast}, y^{\ast}) $ like $ \mathcal{O}(\theta^{k}) $, with $ 0 < \theta < 1 $, as $ k \to + \infty $;

			\item linear convergence of the minimax gap $ \Psi(\bar{x}_{K},y^{\ast}) - \Psi(x^{\ast},\bar{y}_{K}) $ to zero like $ \mathcal{O}(\theta^{K}) $ as $K \to + \infty $.
		\end{itemize}
	\end{enumerate}
\end{enumerate}

\section{Preliminaries}\label{sec:prelim}

We recall some basic notions in convex analysis and monotone operator theory (see for example \cite{BC}). The real Hilbert spaces $ \cH $ and $ \cG $ are endowed with inner products $ \sprod{\ph}{\ph}_{\cH} $ and $ \sprod{\ph}{\ph}_{\cG} $, respectively. As it will be clear from the context  which one is meant, we will drop the index for ease of notation and write $ \sprod{\ph}{\ph} $ for both. The norm induced by the respective inner products is defined by $ \norm{\ph} := \sqrt{\sprod{\ph}{\ph}} $.

A function $ f: \cH \to  \R \cup \{+\infty \}$ is said to be \emph{proper} if  $\dom f := \{x \in \cH: f(x) < + \infty \} \neq\emptyset$. The \emph{(convex) subdifferential} of the function $f:\cH \to \R \cup \{ + \infty \}$ at $ x \in \cH$ is defined by  $\partial f (x) := \{ u \in \cH \hspace{2mm} \vert \, \sprod{y - x}{u} + f(x) \leq f(y) \ \forall y \in \cH \}$ if $f(x) \in  \R$ and by $\partial f (x) := \emptyset$ otherwise. If the function $f$ is convex and Fr{\'e}chet differentiable at $x \in \cH $, then $ \partial f (x) = \{\nabla f (x)\} $. For the sum of a proper, convex and lower semicontinuous function $ f: \cH \to  \R \cup \{+\infty \}$ and a convex and Fr{\'e}chet differentiable function $ h: \cH \to \R $ we have $ \partial (f + h)(x) = \partial f(x) + \nabla h(x) $ for all $ x \in \cH $. The subdifferential of the \emph{indicator function} $\delta_C$ of a nonempty closed convex set $ C \subseteq \cH$,  defined as $\delta_C(x) = 0$ for $x \in C$ and $\delta_C(x) = +\infty$ otherwise, is denoted by $N_C:=\partial \delta_C$ and is called the \emph{normal cone} the set $C$.

Let $ f: \cH \to \R \cup \{ + \infty \} $ be proper, convex and lower semicontinuous. The \emph{proximal operator} of $f$ is defined by
	\begin{equation}
		\label{prox}
		\textup{prox}_{f}: \cH \to \cH, \quad \prox{f}{x} := \argmin_{y \in \cH} \left\{ f(y) + \frac{1}{2} \normsq{y - x} \right\}.
	\end{equation}
The proximal  operator of the indicator function $\delta_C$ of a nonempty closed convex set $ C \subseteq \cH$ is the \emph{orthogonal projection} $P_C : \cH \rightarrow C$ of the set $C$.

A set-valued operator $ A: \cH \rightrightarrows \cH $ is said to be \emph{monotone} if for all $ (x, u) $, $ (y, v) \in \gra A := \{ (z, w) \in \cH \times \cH \: \vert \: w \in A z \} $ we have $\sprod{x - y}{u - v} \geq 0$. Furthermore, $ A $ is said to be  \emph{maximal monotone} if it is monotone and  there exists no monotone operator $ B: \cH \rightrightarrows \cH $ such that $ \gra A \subsetneqq \gra B$. The graph of a maximal monotone operator $ A: \cH \rightrightarrows \cH $ is \emph{sequentially closed} in the \emph{strong} $\times$ \emph{weak} topology, which means that if $ (x_{k}, u_{k})_{k \geq 0} $ is a sequence in $ \gra A $ such that $ x_{k} \to x $ and $ u_{k} \rightharpoonup u$ as $k \rightarrow +\infty$, then  $ (x, u) \in \gra A$. The notation $ u_{k} \rightharpoonup u $ as $k \rightarrow +\infty$ is used to denote convergence of the sequence $(u_k)_{k \geq 0}$ to $u$ in the weak topology.

To show weak convergence of sequences in Hilbert spaces we use the following so-called \emph{Opial Lemma}.

\begin{lemma}\label{lem:opial} (Opial Lemma \cite{opial})
Let $ C \subseteq \cH $ be a nonempty set and $(x_{k})_{k \geq 0}$ a sequence in $ \cH $ such that the following two conditions hold:
	\begin{itemize}
		\item[(a)] for every $ x \in C $, $ \lim_{k \to + \infty} \norm{x_{k} - x} $ exists;

		\item[(b)] every weak sequential cluster point of $ (x_{k})_{k \geq 0} $ belongs to $C$.
	\end{itemize}
	Then $ (x_{k})_{k \geq 0} $ converges weakly to an element in $C$.
\end{lemma}

In the following definition we adjust the term \emph{proper} to the saddle point setting and refer to \cite{TR} for further considerations related to saddle functions.

\begin{definition}\label{def-rock}
	A function $ \Psi : \cH \times \cG \to \R \cup \{ \pm \infty \} $ is called a \emph{saddle function}, if $ \Psi (\ph, y) $ is convex for all $ y \in \cG $ and $ \Psi (x, \ph) $ is concave for all $ x \in \cG $. A saddle function $ \Psi $ is called \emph{proper}, if there exists $ (x', y') \in \cH \times \cG $ such that $ \Psi (x', y) < + \infty $ for all $ y \in \cG $ and $ - \infty < \Psi (x, y') $ for all $ x \in \cH $.
\end{definition}

We conclude the preliminary section with a useful result regarding the minimax objective from~\eqref{min-max}.

\begin{proposition}\label{prop:saddle-function}
The function $ \Psi: \cH \times \cG \to \R \cup \{ \pm \infty \} $ defined via~\eqref{psi} is a proper saddle function such that $ \Psi (\ph, y) $ is lower semicontinuous for each $ y \in \cG $ and $ \Psi (x, \ph) $ is upper semicontinuous for each $ x \in \cH $. Consequently, the operator
$$(x,y) \mapsto
		\partial [\Psi(\ph,y)](x)
		\times
		\partial [\minus \Psi(x, \ph)](y)$$
	is maximal monotone.
\end{proposition}

\begin{proof}
We choose $(x', y') \in \cH \times \cG $ and distinguish four cases.

Firstly, we look at the case $ y' \notin \dom g $. Then
	\begin{equation}
		\Psi(x, y') =
		\begin{cases}
			- \infty & \text{if } x \in \pr_{\cH} (\dom \Phi),\\
			+ \infty & \text{if } x \notin \pr_{\cH} (\dom \Phi),
		\end{cases}
	\end{equation}
thus $x \mapsto \Psi(x,y')$ is convex and lower semicontinuous, since $ \pr_{\cH} (\dom \Phi) $ is convex and closed. Secondly, if $ y' \in \dom g $, then $ g(y') \in \R $ and
	\begin{equation}
		\Psi(x, y') = \Phi(x,y') - g(y') \ \forall x \in \cH,
	\end{equation}
which means that $x \mapsto \Psi(x,y')$ is convex and lower semicontinuous. This proves that $ \Psi (\ph, y) $ is convex and lower semicontinuous for all $ y \in \cG $.

On the other hand, if $ x' \notin \pr_{\cH} (\dom \Phi) $, then
	\begin{equation}
		\Psi(x', y) = + \infty \ \forall y \in \cG,
	\end{equation}
	which means that $y \mapsto \Psi (x', y) $  is upper semicontinuous and concave. Finally, if $ x' \in \pr_{\cH} (\dom \Phi) $, then
	\begin{equation}
		 \Phi (x', y) \in \R \ \mbox{and} \ -\Psi(x', y) = -\Phi(x',y) + g(y) \ \forall y \in \cG.
	\end{equation}
Hence $y \mapsto -\Psi (x', y) $ is proper, convex and lower semicontinuous, and so $ y \mapsto \Psi (x', y) $ is concave and upper semicontinuous. This proves that $ \Psi (x, \ph) $ is concave and upper semicontinuous for all $ x \in \cH $.

Moreover, $ \Psi $ is a proper saddle function. By assumption we have $ g(y) > - \infty $ for all $ y \in \cG $ and there exists $ x' \in \pr_{\cH} (\dom \Phi) \neq \emptyset $ such that $ \dom \Phi (x', \ph) = \cG $. Thus
	\begin{equation}
		\Psi(x', y) = \Phi (x', y) - g(y) < + \infty \ \forall y \in \cG.
	\end{equation}
	Furthermore, by assumption there exist $ y' \in \dom g \subseteq \cG $ such that $ g(y') < + \infty $ and for all $ x \in \cH $ we have $ \Phi (x, y') > - \infty $. Hence,
	\begin{equation}
		\Psi(x, y') = \Phi (x, y') - g(y') > - \infty \ \forall x \in \cH.
	\end{equation}
The maximal monotonicity of $(x,y) \mapsto
		\partial [\Psi(\ph,y)](x)
		\times
		\partial [\minus \Psi(x, \ph)](y)$
 follows from Corollary 1 and Theorem 3 in \cite[pages 248-249]{TR}.
\end{proof}

\section{Convex-(strongly) concave setting}\label{sec:convex-strconcave}

First we will treat the case when the coupling function $ \Phi $ is convex-concave and $g$ is convex with modulus $ \nu \geq 0 $. In the case $ \nu = 0 $ this corresponds to $ \Psi(x,y) = \Phi(x,y) - g(y) $ being convex-concave, while for $ \nu > 0 $ the saddle function $ \Psi(x, y) $ is convex-strongly concave.

We will start with stating two assumptions on the step sizes of the algorithm which will be needed in the convergence analysis. These will be followed by a unified preparatory analysis for general $ \nu \geq 0 $ that will be the base to show convergence of the iterates as well as of the minimax gap. After that we will introduce a choice of parameters that satisfy the aforementioned assumptions. The section will be closed by convergence results for the convex-concave ($ \nu = 0 $) and the convex-strongly concave ($ \nu > 0 $) setting.

\begin{assumption}\label{A1}
	We assume that the step sizes $ \tau_{k} $, $ \sigma_{k} $ and the momentum parameter $ \theta_{k} $ satisfy
	\begin{equation}
		\label{stepsize-tele-nice}
		\tau_{k+1} \geq \frac{\tau_{k}}{\theta_{k+1}} \quad \mbox{and} \quad
		\sigma_{k+1} \geq \frac{\sigma_{k}}{\theta_{k+1} (1 + \nu \sigma_{k})}
		\quad
		\text{for all } k \geq 0.
	\end{equation}
	Furthermore we assume that there exist $ \delta > 0 $ and $ (\alpha_{k})_{k \geq 0} \subseteq \R_{++}$ such that
	\begin{equation}
		\label{stepsize-delta}
		\frac{1 - \delta}{\tau_{k}} \geq \frac{L_{yx}}{\alpha_{k+1}}
		\quad \mbox{and} \quad
		\frac{1 - \delta}{\sigma_{k}} \geq L_{yx} \alpha_{k} \theta_{k} + L_{yy} (1 + \theta_{k})
		\quad
		\text{for all } k \geq 0,
	\end{equation}
	where $ \theta_{0} := 1 $.
\end{assumption}

\subsection{Preliminary considerations}

In this subsection we will make some preliminary considerations that will play an important role when proving the convergence properties of the numerical scheme given by~\eqref{alg_y}-\eqref{alg_x}. For all $ k \geq 0 $ we will use the notations
\begin{equation}		\label{not_q-v}
		q_{k} := \ny \Phi(x_{k}, y_{k}) - \ny \Phi(x_{k-1}, y_{k-1}) \ \mbox{and} \ s_{k} := \theta_{k} q_{k} + \ny \Phi (x_{k}, y_{k}).
\end{equation}
We take an arbitrary $(x,y) \in \cH \times \cG$ and let $ k \geq 0$ be fixed.
From~\eqref{alg_y} we derive
\begin{equation}
	\label{opt-cond-y}
	0 \in \partial g(y_{k+1}) + \frac{1}{\sigma_{k}} (y_{k+1} - y_{k}) - s_{k},
\end{equation}
and, as $ g $ is convex with modulus $ \nu $, this implies
\begin{equation}
	\label{subdiff-g}
	\begin{split}
		g(y)
		&\geq g(y_{k+1}) + \sprod{s_{k}}{y - y_{k+1}} + \frac{1}{\sigma_{k}} \sprod{y_{k} - y_{k+1}}{y - y_{k+1}} + \frac{\nu}{2} \normsq{y - y_{k+1}}\\
		&= g(y_{k+1}) + \sprod{s_{k}}{y - y_{k+1}} + \frac{1}{2\sigma_{k}} \left(\normsq{y_{k} - y_{k+1}} + \normsq{y - y_{k+1}} - \normsq{y - y_{k}}\right) + \frac{\nu}{2} \normsq{y - y_{k+1}}.
	\end{split}
\end{equation}
From~\eqref{alg_x} we get
\begin{equation}
	\label{opt-cond-x}
	0 \in \partial [\Phi(\ph, y_{k+1})](x_{k+1}) + \frac{1}{\tau_{k}} (x_{k+1} - x_{k}),
\end{equation}
hence the convexity of $ \Phi (\ph, y) $ for $ y \in \dom g $ yields
\begin{equation}
	\label{subdiff-phi}
	\begin{split}
		\Phi(x,y_{k+1})
		& \geq \Phi(x_{k+1}, y_{k+1}) + \frac{1}{\tau_{k}} \sprod{x_{k} - x_{k+1}}{x - x_{k+1}}\\
		& = \Phi(x_{k+1}, y_{k+1}) + \frac{1}{2\tau_{k}} \left(\normsq{x_{k} - x_{k+1}} + \normsq{x - x_{k+1}} - \normsq{x - x_{k}}\right).
	\end{split}
\end{equation}
Combining~\eqref{subdiff-g} and~\eqref{subdiff-phi} we obtain
\begin{equation}
	\begin{split}
		\Psi(x_{k+1},y)-\Psi(x,y_{k+1}) &= \Phi(x_{k+1}, y) - g(y) - \Phi(x, y_{k+1}) + g(y_{k+1})\\
			&\leq \Phi(x_{k+1}, y) - \Phi(x, y_{k+1})
			+ \sprod{s_{k}}{y_{k+1} - y}
			- \frac{\nu}{2} \normsq{y - y_{k+1}}\\
				& \hspace{5mm}+ \frac{1}{2\sigma_{k}} \left( \minus \normsq{y_{k} - y_{k+1}} - \normsq{y - y_{k+1}} + \normsq{y - y_{k}} \right)\\
			&\leq \Phi(x_{k+1}, y) - \Phi(x_{k+1}, y_{k+1})
			+ \sprod{s_{k}}{y_{k+1} - y}
			- \frac{\nu}{2} \normsq{y - y_{k+1}}\\
				&\hspace{5mm}+ \frac{1}{2\tau_{k}} \left( \minus \normsq{x_{k} - x_{k+1}} - \normsq{x - x_{k+1}} + \normsq{x - x_{k}} \right)\\
				&\hspace{5mm}+ \frac{1}{2\sigma_{k}} \left( \minus \normsq{y_{k} - y_{k+1}} - \normsq{y - y_{k+1}} + \normsq{y - y_{k}} \right),
	\end{split}
\end{equation}
which, together with the concavity of $\Phi$ in the second variable and~\eqref{not_q-v}, gives
\begin{equation}
	\label{comb}
	\begin{split}
		\Psi(x_{k+1},y)-\Psi(x,y_{k+1})
			\leq & \ \theta_{k} \sprod{q_{k}}{y_{k+1}-y} - \frac{\nu}{2} \normsq{y - y_{k+1}}\\
				& - \sprod{\ny \Phi(x_{k+1}, y_{k+1})}{y_{k+1} - y} + \sprod{\ny \Phi(x_{k}, y_{k})}{y_{k+1} - y}\\
				& + \frac{1}{2\tau_{k}} \left( \minus \normsq{x_{k} - x_{k+1}} - \normsq{x - x_{k+1}} + \normsq{x - x_{k}} \right)\\
				& + \frac{1}{2\sigma_{k}} \left( \minus \normsq{y_{k} - y_{k+1}} - \normsq{y - y_{k+1}} + \normsq{y - y_{k}} \right)\\
			=& \ \minus \sprod{q_{k+1}}{y_{k+1} - y} + \theta_{k} \sprod{q_{k}}{y_{k} - y} - \frac{\nu}{2} \normsq{y - y_{k+1}}\\
				& + \frac{1}{2\tau_{k}} \left( \minus \normsq{x_{k} - x_{k+1}} - \normsq{x - x_{k+1}} + \normsq{x - x_{k}} \right)\\
				& + \frac{1}{2\sigma_{k}} \left( \minus \normsq{y_{k} - y_{k+1}} - \normsq{y - y_{k+1}} + \normsq{y - y_{k}} \right)\\
				& + \theta_{k} \sprod{q_{k}}{y_{k+1} - y_{k}}.
	\end{split}
\end{equation}
By using~\eqref{cond-Phi_y} we can evaluate the last term in the above expression as follows
\begin{equation}
	\label{estim}
	\begin{split}
	\vert \sprod{q_k}{y-y_{k}} \vert
		&\leq \norm{q_k} \, \norm{y-y_{k}} \leq \left( L_{yx} \norm{x_{k}-x_{k-1}} + L_{yy} \norm{y_{k}-y_{k-1}} \right) \norm{y-y_{k}}\\
		&\leq \frac{L_{yx}}{2} \left( \alpha_{k} \normsq{y-y_{k}} + \frac{1}{\alpha_{k}} \normsq{x_{k}-x_{k-1}} \right) + \frac{L_{yy}}{2} \left( \normsq{y-y_{k}} + \normsq{y_{k}-y_{k-1}} \right),
	\end{split}
\end{equation}
with $ \alpha_{k} > 0 $ chosen such that~\eqref{stepsize-delta} holds.

Writing~\eqref{estim} for $ y := y_{k+1} $ and combining the resulting inequality with~\eqref{comb} we derive
\begin{equation}
	\label{est-gap}
	\Psi(x_{k+1},y) - \Psi(x,y_{k+1}) \leq a_{k}(x, y) - b_{k+1}(x, y) - c_{k},
\end{equation}
where
\begin{equation}
	\begin{split}
		a_k(x, y)
		& :=
		\frac{1}{2\tau_{k}} \normsq{x-x_{k}}
		+ \frac{1}{2\sigma_{k}} \normsq{y-y_{k}}
		+ \theta_{k} \sprod{q_k}{y_{k}-y}
		+ \theta_{k} \frac{L_{yx}}{2\alpha_{k}} \normsq{x_{k}-x_{k-1}}\\
		&\hspace{5mm}+ \theta_{k} \frac{L_{yy}}{2} \normsq{y_{k}-y_{k-1}},
	\end{split}
\end{equation}
\begin{equation}
	\begin{split}
		b_{k+1}(x, y)
		& :=
		\frac{1}{2\tau_{k}} \normsq{x-x_{k+1}}
		+ \frac{1}{2} \left( \frac{1}{\sigma_{k}} + \nu \right) \normsq{y-y_{k+1}}
		+ \sprod{q_{k+1}}{y_{k+1} - y}\\
		&\hspace{5mm}+ \frac{L_{yx}}{2\alpha_{k+1}} \normsq{x_{k+1}-x_{k}}
		+ \frac{L_{yy}}{2} \normsq{y_{k+1} - y_{k}},
	\end{split}
\end{equation}
and
\begin{equation}
	\begin{split}
		c_{k}
		& :=
		\frac{1}{2} \left( \frac{1}{\tau_{k}} - \frac{L_{yx}}{\alpha_{k+1}} \right) \normsq{x_{k+1}-x_{k}}
		+ \frac{1}{2} \left( \frac{1}{\sigma_{k}} - L_{yy} - \theta_{k} (L_{yx} \alpha_{k} + L_{yy}) \right) \normsq{y_{k+1} - y_{k}}.
	\end{split}
\end{equation}

Now, let us define for all $k \geq 0$
\begin{equation}
	\label{tk}
	t_{k} := \frac{\theta_{0}}{\theta_{0} \theta_{1} \cdots \theta_{k}}
\end{equation}
and notice that
\begin{equation}
	\frac{t_{k}}{t_{k+1}} = \theta_{k+1}.
\end{equation}
Relation~\eqref{stepsize-tele-nice} from Assumption~\ref{A1} is equivalent to
\begin{equation}
	\label{stepsize-tele-ugly}
	\frac{t_{k}}{\tau_{k}} \geq \frac{t_{k+1}}{\tau_{k+1}} \quad \mbox{and} \quad
	t_{k}\left( \frac{1}{\sigma_{k}} + \nu \right) \geq \frac{t_{k+1}}{\sigma_{k+1}},
\end{equation}
which will be used in telescoping arguments in the following.

Let $ K \geq 1 $ and denote
\begin{equation}
	\label{erg-sequ}
	T_{K} := \sum_{k=0}^{K-1} t_{k},
	\hspace{5mm}
	\bar{x}_{K} := \frac{1}{T_{K}} \sum_{k=0}^{K-1} t_{k}x_{k+1},
	\hspace{5mm}
	\bar{y}_{K} := \frac{1}{T_{K}} \sum_{k=0}^{K-1} t_{k}y_{k+1}.
\end{equation}
Multiplying both sides of~\eqref{est-gap} by $ t_{k} > 0 $ as defined in~\eqref{tk}, followed by summing up the inequalities for~$ k = 0, \ldots, K-1 $ gives
\begin{equation}
	\sum_{k=0}^{K-1} t_{k} \left( \Psi(x_{k+1}, y) - \Psi(x, y_{k+1}) \right)
	\leq \sum_{k=0}^{K-1} t_{k} \left( a_{k}(x, y) - b_{k+1}(x, y) - c_{k} \right).
\end{equation}
By Jensen's inequality, as $ \Psi(\ph, y) - \Psi(x, \ph) $ is a convex function, we obtain
\begin{equation}
	T_{K} \left( \Psi(\bar{x}_{K},y) - \Psi(x,\bar{y}_{K}) \right)
	\leq \sum_{k=0}^{K-1} t_{k} \left( \Psi(x_{k+1}, y) - \Psi(x, y_{k+1}) \right),
\end{equation}
and thus
\begin{equation}
	\label{est-gap-erg-aux}
	T_{K} \left( \Psi(\bar{x}_{K},y) - \Psi(x,\bar{y}_{K}) \right)
	\leq \sum_{k=0}^{K-1} t_{k} \left( a_{k}(x, y) - b_{k+1}(x, y) - c_{k} \right).
\end{equation}
Furthermore, using~\eqref{stepsize-tele-ugly}, we get for all $k \geq 0$
\begin{equation}
	\label{akbk}
	\begin{split}
		t_{k} b_{k+1}(x, y)
		&= \frac{t_{k}}{2\tau_{k}} \normsq{x-x_{k+1}}
			+ \frac{t_{k}}{2} \left( \frac{1}{\sigma_{k}} + \nu \right) \normsq{y-y_{k+1}}
			+ t_{k} \sprod{q_{k+1}}{y_{k+1} - y}\\
			& \hspace{5mm} + t_{k} \frac{L_{yx}}{2\alpha_{k+1}} \normsq{x_{k+1}-x_{k}}
			+ t_{k} \frac{L_{yy}}{2} \normsq{y_{k+1} - y_{k}}\\
		& \geq \frac{t_{k+1}}{2\tau_{k+1}} \normsq{x-x_{k+1}}
			+ \frac{t_{k+1}}{2\sigma_{k+1}} \normsq{y-y_{k+1}}
			+ t_{k+1} \theta_{k+1} \sprod{q_{k+1}}{y_{k+1} - y}\\
			& \hspace{5mm} + t_{k+1} \theta_{k+1} \frac{L_{yx}}{2\alpha_{k+1}} \normsq{x_{k+1}-x_{k}}
			+ t_{k+1} \theta_{k+1} \frac{L_{yy}}{2} \normsq{y_{k+1} - y_{k}}\\
		& = t_{k+1} a_{k+1}(x, y).
	\end{split}
\end{equation}
Notice that by~\eqref{stepsize-delta} in Assumption~\ref{A1} there exists $ \delta > 0 $ such that for all $k \geq 0$
\begin{equation}
	\label{ck}
	c_{k} \geq \delta
	\left(
		\frac{1}{2\tau_{k}} \normsq{x_{k+1}-x_{k}} +
		\frac{1}{2\sigma_{k}} \normsq{y_{k+1}-y_{k}}
	\right)
	\geq 0.
\end{equation}
For the following recall that $ x_{-1} = x_{0} $ and $ y_{-1} = y_{0} $, which implies $ q_{0} = 0 $. By using the above two inequalities in~\eqref{est-gap-erg-aux} and writing~\eqref{estim} for $ k = K $ we obtain
\begin{equation}
	\label{est-gap-erg-1}
	\begin{split}
		\Psi(\bar{x}_{K},y) - \Psi(x,\bar{y}_{K})
			& \leq \frac{1}{T_{K}} \sum_{k=0}^{K-1} \left( t_{k} a_{k}(x, y) - t_{k+1} a_{k+1}(x, y) \right) = \frac{1}{T_{K}} \left( t_{0} a_{0}(x, y) - t_{K} a_{K}(x, y) \right)\\
		&= \frac{1}{T_{K}} \left( \frac{t_{0}}{2 \tau_{0}} \normsq{x-x_{0}} + \frac{t_{0}}{2 \sigma_{0}} \normsq{y-y_{0}} \right) - \frac{t_{K}}{T_{K}} \left( \frac{1}{2 \tau_{K}} \normsq{x-x_{K}} + \frac{1}{2 \sigma_{K}} \normsq{y-y_{K}} \right)\\
			&\hspace{5mm} - \frac{t_{K} \theta_{K}}{T_{K}} \left( \sprod{q_{K}}{y_{K}-y}
			+ \frac{L_{yx}}{2\alpha_{K}} \normsq{x_{K}-x_{K-1}}
			+ \frac{L_{yy}}{2} \normsq{y_{K}-y_{K-1}} \right)\\
		&\leq \frac{1}{T_{K}} \left( \frac{t_{0}}{2 \tau_{0}} \normsq{x-x_{0}} + \frac{t_{0}}{2 \sigma_{0}} \normsq{y-y_{0}} \right)\\
			&\hspace{5mm} - \frac{t_{K}}{T_{K}} \left( \frac{1}{2 \tau_{K}} \normsq{x-x_{K}} + \frac{1}{2} \left( \frac{1}{\sigma_{K}} - \theta_{K} ( L_{yx} \alpha_{K} + L_{yy}) \right) \normsq{y-y_{K}} \right).
	\end{split}
\end{equation}
By definition we have $ t_{0} = 1 $ and by~\eqref{stepsize-delta} that the last term of the above inequality is nonpositive, hence  the following estimate for the minimax gap function evaluated at the ergodic sequences holds
\begin{equation}
	\label{est-gap-erg-2}
	\Psi(\bar{x}_{K},y) - \Psi(x,\bar{y}_{K})
	\leq \frac{1}{T_{K}} \left( \frac{1}{2 \tau_{0}} \normsq{x-x_{0}} + \frac{1}{2 \sigma_{0}} \normsq{y-y_{0}} \right) \ \forall K \geq 1.
\end{equation}
With these considerations at hand -- in specific we want to point out~\eqref{est-gap}, \eqref{est-gap-erg-1} and~\eqref{est-gap-erg-2} -- we will be able to obtain convergence statements for the two settings $ \nu = 0 $ and $ \nu > 0 $.

\subsection{Fulfilment of step size assumptions}\label{ssec:params}
In this subsection we will investigate a particular choice of parameters to fulfil Assumption~\ref{A1} which is suitable for both cases of $ \nu = 0 $ and  $\nu > 0$.

\begin{proposition}\label{prop:stepsize}
	Let $ \nu \geq 0 $, $ c_{\alpha} > L_{yx} \geq 0 $, $ \theta_{0} = 1 $ and  $ \tau_{0}, \, \sigma_{0} > 0 $ such that
	\begin{equation}
		\label{initial-stepsizes}
		\left( c_{\alpha} L_{yx} \tau_{0} + 2 L_{yy} \right) \sigma_{0} < 1.
	\end{equation}
	We define
	\begin{equation}
		\label{stepsize-pdhg}
		\theta_{k+1} := \frac{1}{\sqrt{1 + \nu \sigma_{k}}},
			\hspace{5mm}
		\tau_{k+1} := \frac{\tau_{k}}{\theta_{k+1}},
			\hspace{5mm}
		\sigma_{k+1} := \theta_{k+1} \sigma_{k}
		\quad
		\text{for all } k \geq 0.
	\end{equation}
Then the sequence  $ (\tau_{k})_{k \geq 0} $, $ (\sigma_{k})_{k \geq 0} $ and $ (\theta_{k})_{k \geq 0} $ fulfil~\eqref{stepsize-tele-nice} in Assumption~\ref{A1} with equality and~\eqref{stepsize-delta} for
	\begin{equation}
		\label{alpha}
		\alpha_{k} :=
		\left\{
			\begin{array}{ll}
				c_{\alpha} \tau_{0} & \text{if } k = 0, \vspace{2mm}\\
				c_{\alpha} \tau_{k-1} & \text{if } k \geq 1,
			\end{array}
		\right.
	\end{equation}
and
	\begin{equation}
		\label{delta}
		\delta := \min
		\left\{
			1 - \frac{L_{yx}}{c_{\alpha}}, \,
			1 - \left( c_{\alpha} L_{yx} \tau_{0} + 2 L_{yy} \right) \sigma_{0}
		\right\} > 0.
	\end{equation}
	Furthermore, for $(t_k)_{k \geq 0}$ defined as in ~\eqref{tk} we have
	\begin{equation}
		\label{tk-pdhg}
		t_{k} = \frac{\theta_{0}}{\theta_{0} \theta_{1} \cdots \theta_{k}} =  \frac{\tau_{k}}{\tau_{0}} \quad \forall k \geq 0.
	\end{equation}
\end{proposition}

\begin{proof}
	First, we show that the particular choice~\eqref{stepsize-pdhg} fulfils~\eqref{stepsize-tele-nice} in Assumption~\ref{A1} with equality.
	We see that for all $ k \geq 0 $
	\begin{equation}
	\tau_{k+1} = \frac{\tau_{k}}{\theta_{k+1}},
	\end{equation}
	as well as
	\begin{equation}
	\sigma_{k+1}
	= \theta_{k+1} \sigma_{k}
	= \frac{\sigma_{k}}{\theta_{k+1} \frac{1}{\theta_{k+1}^{2}}}
	= \frac{\sigma_{k}}{\theta_{k+1} (1 + \nu \sigma_{k})},
	\end{equation}
	follow straight forward by definition.

	Next, we show that~\eqref{stepsize-delta} in Assumption~\ref{A1} holds for $\delta$ defined in \eqref{delta}
	with the choices~\eqref{stepsize-pdhg} and~\eqref{alpha}.
	The first inequality of~\eqref{stepsize-delta} is equivalent to
	\begin{equation}
		1 - \delta \geq \frac{L_{yx}}{\alpha_{k+1}} \tau_{k} = \frac{L_{yx}}{c_{\alpha}} \quad \forall k \geq 0,
	\end{equation}
	which clearly is fulfilled as
	\begin{equation}
		\delta \leq 1 - \frac{L_{yx}}{c_{\alpha}}.
	\end{equation}
	On the other hand, the second inequality of~\eqref{stepsize-delta} is equivalent to
	\begin{equation}
		1 - \delta
		\geq L_{yx} \alpha_{k} \theta_{k} \sigma_{k} + L_{yy} (1 + \theta_{k}) \sigma_{k} \quad \forall k \geq 0.
	\end{equation}
	By definition of the step size parameters~\eqref{stepsize-pdhg} we have for all $k \geq 0$
	\begin{equation}
		\tau_{k+1} \sigma_{k+1} = \tau_{0} \sigma_{0},
		\hspace{5mm}
		\theta_{k+1} \leq 1 = \theta_{0},
		\hspace{5mm}
		\sigma_{k+1} \leq \sigma_{0},
		\hspace{5mm}
		\theta_{k+1} \tau_{k+1} = \tau_{k},
	\end{equation}
	and thus
	\begin{equation}
		\begin{split}
			1 - \delta
			&\geq L_{yx} \alpha_{0} \theta_{0} \sigma_{0} + L_{yy} (1 + \theta_{0}) \sigma_{0}
			= c_{\alpha} L_{yx} \tau_{0} \sigma_{0} + 2 L_{yy} \sigma_{0} \geq c_{\alpha} L_{yx} \theta_{k+1}^{2} \tau_{k+1} \sigma_{k+1} + L_{yy} (1 + \theta_{k+1}) \sigma_{k+1}\\ &= L_{yx} c_{\alpha} \tau_{k} \theta_{k+1} \sigma_{k+1} + L_{yy} (1 + \theta_{k+1}) \sigma_{k+1} = L_{yx} \alpha_{k+1} \theta_{k+1} \sigma_{k+1} + L_{yy} (1 + \theta_{k+1}) \sigma_{k+1}.
		\end{split}
	\end{equation}
	This chain of inequalities holds since
	\begin{equation}
		\delta
		\leq 1 - \left( c_{\alpha} L_{yx} \tau_{0} + 2 L_{yy} \right) \sigma_{0}.
	\end{equation}
Finally, using the definition of $ t_{k} $ and~\eqref{stepsize-pdhg} we conclude that for all $ k \geq 0 $
	\begin{equation}
		t_{k}
		= \frac{\theta_{0}}{\theta_{0} \theta_{1} \cdots \theta_{k}}
		= \frac{\frac{\tau_{0}}{\tau_{0}}}{ \frac{\tau_{0}}{\tau_{0}} \frac{\tau_{0}}{\tau_{1}} \cdots \frac{\tau_{k-1}}{\tau_{k}} }
		= \frac{\tau_{k}}{\tau_{0}}.
	\end{equation}
\end{proof}

\begin{remark}
The choice $L_{yy} = 0$ in \eqref{cond-Phi_y} which was considered in \cite{Aybat} in the convex-strongly concave setting corresponds to the case when the coupling function $\Phi$ is linear in $y$. We will prove convergence also for $L_{yy}$ positive, which makes our algorithm applicable to a much wider range of problems, as we will see in the section with the numerical experiments.

When the coupling function $ \Phi: \cH \times \cG \to \R $ is bilinear, that is $ \Phi (x, y) = \sprod{y}{Ax} $ for some nonzero continuous linear operator $ A : \cH \to \cG $  then we are in the setting of~\cite{PDHG}. In this situation one can choose $ L_{yy} = 0 $ and $L_{yx} = \|A\|$, and  \eqref{delta} yields
	\begin{equation}
		\label{delta-Lyy0}
		\delta = \min
		\left\{
			1 - \frac{\|A\|}{c_{\alpha}}, \,
			1 - c_{\alpha} \|A\| \tau_{0} \sigma_{0}
		\right\},
	\end{equation}
	with $ c_{\alpha} > \|A\| $.
	To guarantee $ \delta > 0 $ we fix $ 0 < \varepsilon < 1 $ and set
	\begin{equation}
		c_{\alpha} = (1 - \varepsilon)^{-1} \|A\|.
	\end{equation}
	Hence, we need to satisfy
	\begin{equation}
		\tau_{0} \sigma_{0}\|A\|^{2} < 1 - \varepsilon,
	\end{equation}
	which heavily resembles the step size condition of~\cite[Algorithm~2]{PDHG}. Since $\prox{\gamma \Phi(\cdot, y)}{x} = x - \gamma A^*y$ for all $(x,y)\in \cH \times \cG$ and all $\gamma > 0$, our OGAProx scheme becomes the primal-dual algorithm PDHG from \cite{PDHG}.

% 	\begin{equation}
% 		\tau_{0} \sigma_{0} < \frac{1}{L^{2}},
% 	\end{equation}
% 	with $ L := \norm{A} = L_{yx} $.
\end{remark}

\subsection{Convergence results}
In this subsection we combine the preliminary considerations with the choice of parameters~\eqref{stepsize-pdhg} from Proposition~\ref{prop:stepsize}.

We will start with the case $ \nu = 0 $ and constant step sizes, which gives weak convergence of the iterates to a saddle point $ (x^{\ast}, y^{\ast}) $ and convergence of the minimax gap evaluated at the ergodic iterates to zero like $ \mathcal{O}(\frac{1}{K}) $.  Afterwards we will consider the case $ \nu > 0 $, which leads to an accelerated version of the algorithm with improved convergence results. In this setting we obtain convergence of $ (y_{k})_{k \geq 0} $ to $ y^{\ast} $ like  $ \mathcal{O}(\frac{1}{K}) $ and convergence of the minimax gap evaluated at the ergodic iterates to zero like $ \mathcal{O}(\frac{1}{K^{2}}) $.

\subsubsection{Convex-concave setting}\label{convex-concave-setting}
For the following we assume that the function $ g $ is convex with modulus $ \nu = 0 $, meaning it is merely convex. Using the results of the previous subsection we will show that with the choice~\eqref{stepsize-pdhg} all the parameters are constant.

\begin{proposition}\label{cor:stepsize-convex-concave}
	Let $ c_{\alpha} > L_{yx} \geq 0 $ and $ \tau, \, \sigma > 0 $ such that
	\begin{equation}
		\left( c_{\alpha} L_{yx} \tau + 2 L_{yy} \right) \sigma < 1.
	\end{equation}
	If $ \nu = 0 $, then the sequences $ (\tau_{k})_{k \geq 0} $, $ (\sigma_{k})_{k \geq 0} $ and $ (\theta_{k})_{k \geq 0} $ as defined in Proposition~\ref{prop:stepsize} are constant, in particular we have
	\begin{equation}
		\label{const-param}
		\tau_{k} = \tau_{0} := \tau,
			\hspace{5mm}
		\sigma_{k} = \sigma_{0} := \sigma,
			\hspace{5mm}
		\theta_{k} = \theta_{0} = 1
		\quad
		\text{for all } k \geq 0.
	\end{equation}
\end{proposition}
\begin{proof}
	As $ \nu = 0 $, \eqref{stepsize-pdhg} gives for all $ k \geq 0 $
	\begin{equation}
		\theta_{k+1} = \frac{1}{\sqrt{1 + \nu \sigma_{k}}} = 1,
		\hspace{5mm}
		\tau_{k+1} = \frac{\tau_{k}}{\theta_{k+1}} = \tau_{0},
		\hspace{5mm}
		\sigma_{k+1} = \theta_{k+1} \sigma_{k} = \sigma_{0}.
	\end{equation}
\end{proof}

Next we will state and prove the convergence results in the convex-concave case.
\begin{theorem}\label{convex-concave}
	Let $ c_{\alpha} > L_{yx} \geq 0 $ and $ \tau, \, \sigma > 0 $ such that
	\begin{equation}
		\left( c_{\alpha} L_{yx} \tau + 2 L_{yy} \right) \sigma < 1.
	\end{equation}
	Then the sequence $ (x_{k}, y_{k})_{k \geq 0} $ generated by OGAProx with the choice of constant parameters as in Proposition \ref{cor:stepsize-convex-concave}, namely,
	\begin{equation}
		\tau_{k} = \tau_{0} := \tau,
			\hspace{5mm}
		\sigma_{k} = \sigma_{0} := \sigma,
			\hspace{5mm}
		\theta_{k} = \theta_{0} = 1
		\quad
		\text{for all } k \geq 0,
	\end{equation}
	converges weakly to a saddle point $ (x^{\ast},y^{\ast}) \in \cH \times \cG $ of~\eqref{min-max}. Furthermore, let $ K \geq 1 $ and denote
	\begin{equation}
		\bar{x}_{K} = \frac{1}{K} \sum_{k=0}^{K-1} x_{k+1} \quad \mbox{and} \quad
		\bar{y}_{K} = \frac{1}{K} \sum_{k=0}^{K-1} y_{k+1}.
	\end{equation}
	Then for all $ K \geq 1 $ and any saddle point $ (x^{\ast},y^{\ast}) \in \cH \times \cG $ of~\eqref{min-max} we have
	\begin{equation}
		0
		\leq \Psi(\bar{x}_{K},y^{\ast}) - \Psi(x^{\ast},\bar{y}_{K})
		\leq \frac{1}{K} \left( \frac{1}{2 \tau_{0}} \normsq{x^{\ast}-x_{0}} + \frac{1}{2 \sigma_{0}} \normsq{y^{\ast}-y_{0}} \right).
	\end{equation}
\end{theorem}

\begin{proof}
	First we will show weak convergence of the sequence of iterates $ (x_{k}, y_{k})_{k \geq 0} $ to some saddle point $ (x^{\ast},y^{\ast}) \in \cH \times \cG $ of~\eqref{min-max}. For this we will use the Opial Lemma (see Lemma~\ref{lem:opial}).

	Let $ k \geq 0 $ and $ (x^{\ast},y^{\ast}) \in \cH \times \cG $ be an arbitrary but fixed saddle point. From~\eqref{est-gap} together with the choice~\eqref{const-param} of constant parameters $ \theta_{k} = 1 $, $ \tau_{k} = \tau $, $ \sigma_{k} = \sigma $ and $ \alpha_{k} = \alpha $ we obtain
	\begin{equation}
		\label{est-gap-sp}
		\begin{split}
			0
			\leq \Psi(x_{k+1}, y^{\ast}) - \Psi(x^{\ast}, y_{k+1})
			\leq a_{k}(x^{\ast}, y^{\ast}) - b_{k+1}(x^{\ast}, y^{\ast}) - c_{k} = a_{k}(x^{\ast}, y^{\ast}) - a_{k+1}(x^{\ast}, y^{\ast}) - c_{k},
		\end{split}
	\end{equation}
	since
	\begin{equation}
		\label{astar}
		\begin{split}
			a_{k}(x^{\ast},y^{\ast})
			& =
			\frac{1}{2\tau} \normsq{x^{\ast} - x_{k}}
			+ \frac{1}{2\sigma} \normsq{y^{\ast} - y_{k}}
			+ \sprod{q_k}{y_{k} - y^{\ast}}
			+ \frac{L_{yx}}{2\alpha} \normsq{x_{k} - x_{k-1}} + \frac{L_{yy}}{2} \normsq{y_{k} - y_{k-1}}\\
			&= b_{k}(x^{\ast}, y^{\ast}),
		\end{split}
	\end{equation}
	and
	\begin{equation}
		\begin{split}
			c_{k}
			& =
			\frac{1}{2} \left( \frac{1}{\tau} - \frac{L_{yx}}{\alpha} \right) \normsq{x_{k+1} - x_{k}}
			+ \frac{1}{2} \left( \frac{1}{\sigma} - L_{yy} - (L_{yx} \alpha + L_{yy} ) \right) \normsq{y_{k+1} - y_{k}}.
		\end{split}
	\end{equation}
	We see that~\eqref{astar}, writing~\eqref{estim} with $ y = y^{\ast} $ and~\eqref{stepsize-tele-nice} in Assumption~\ref{A1} yield
	\begin{equation}
		\label{astar-nneg}
		a_{k}(x^{\ast}, y^{\ast})
		\geq \frac{1}{2\tau} \normsq{x^{\ast} - x_{k}}
			+ \frac{1}{2\sigma} \left( 1 - \sigma ( L_{yx} \alpha + L_{yy} ) \right) \normsq{y^{\ast} - y_{k}}
		\geq 0.
	\end{equation}
	Furthermore, from~\eqref{est-gap-sp} and~\eqref{ck} we deduce
	\begin{equation}
		a_{k}(x^{\ast}, y^{\ast}) \geq
		a_{k+1}(x^{\ast}, y^{\ast})
		+ \delta
		\left(
			\frac{1}{2\tau} \normsq{x_{k+1} - x_{k}} +
			\frac{1}{2\sigma} \normsq{y_{k+1} - y_{k}}
		\right).
	\end{equation}
	Telescoping this inequality and taking into account~\eqref{astar-nneg} give
	\begin{equation}
		\label{x-x}
		\lim_{k \to + \infty} (x_{k+1} - x_{k})
		= \lim_{k \to + \infty} (y_{k+1} - y_{k})
		= 0,
	\end{equation}
	as well as the existence of the limit $\lim_{k \to + \infty} a_{k}(x^{\ast},y^{\ast}) \in \R$.

	From~\eqref{astar-nneg} we get that $ (x_{k})_{k \geq 0} $ and $ (y_{k})_{k \geq 0} $ are bounded sequences. Moreover, by using~\eqref{cond-Phi_y} and~\eqref{x-x} in definition~\eqref{not_q-v} we obtain that
	\begin{equation}
		\label{q_k-0}
		(q_k)_{k \geq 0} \mbox{ converges strongly to } 0.
	\end{equation}
	From the definition of $ a_{k}(x^{\ast}, y^{\ast}) $ in~\eqref{astar},~\eqref{x-x} and~\eqref{q_k-0} we derive that
	\begin{equation}
		\exists \lim_{k \to + \infty} \left( \frac{1}{2\tau} \normsq{x_{k}-x^{\ast}} + \frac{1}{2\sigma} \normsq{y_{k}-y^{\ast}} \right) \in \R.
	\end{equation}
	Since this is true for an arbitrary saddle point $(x^{\ast},y^{\ast}) \in \cH \times \cG$, we have that the first statement of the Opial Lemma holds.

	Next we will show that all weak cluster points of $ (x_{k},y_{k})_{k \geq 0} $ are in fact saddle points of~\eqref{min-max}. Assume that $ (x_{k_n})_{n \geq 0} $ converges weakly to $ x^{\ast} \in \cH $ and $ (y_{k_n})_{n \geq 0} $ converges weakly to $ y^{\ast} \in \cG $ as $ n \to + \infty $.
	From~\eqref{opt-cond-x},~\eqref{not_q-v} and~\eqref{opt-cond-y} we have
	\begin{equation}
		\label{incl}
		\begin{split}
			&\left( \frac{1}{\tau}(x_{k_n} - x_{k_n+1}),\frac{1}{\sigma}(y_{k_n} - y_{k_n+1}) + q_{k_n} - q_{k_n+1} \right)\\
			&\hspace{2cm} \in
			\partial [\Phi(\ph,y_{k_n+1})](x_{k_n+1}) \times \left( \minus \ny \Phi(x_{k_n+1},y_{k_n+1}) + \partial g(y_{k_n+1}) \right)\\
			&\hspace{2cm} =
			\partial [\Psi(\ph,y_{k_n+1})](x_{k_n+1}) \times \partial [\minus \Psi(x_{k_n+1}, \ph)](y_{k_n+1}),
		\end{split}
	\end{equation}
	where we used that for all $ k \geq 0 $ we have $ x_{k} \in \pr_{\cH} (\dom \Phi) $ and $ y_{k} \in \dom g $.
	The sequence on the left hand side of the inclusion~\eqref{incl} converges strongly to $(0,0)$ as $n\to+\infty$ (according to~\eqref{x-x} and~\eqref{q_k-0}). Notice that the operator
	$(x,y) \mapsto
		\partial [\Psi(\ph,y)](x)
		\times
		\partial [\minus \Psi(x, \ph)](y)$
	is maximal monotone (see Proposition~\ref{prop:saddle-function}), hence its graph
	is sequentially closed with respect to the strong $\times$ weak topology. From here we deduce
	\begin{equation}
		(0,0) \in
		\partial [\Psi(\ph,y^{\ast})](x^{\ast})
		\times
		\partial [\minus \Psi(x^{\ast}, \ph)](y^{\ast}),
	\end{equation}
	from which we easily derive that $(x^{\ast},y^{\ast})$ is a saddle point as it satisfies~\eqref{def-saddle-point}. This means that also the second statement of the Opial Lemma is fulfilled and we have weak convergence of $ (x_{k}, y_{k})_{k \geq 0} $ to a saddle point $ (x^{\ast}, y^{\ast}) $.

The remaining part is to show the convergence rate of the minimax gap of the ergodic sequences. Let $ K \geq 1 $ and $ (x^{\ast},y^{\ast}) \in \cH \times \cG $ be an arbitrary but fixed saddle point. Writing~\eqref{est-gap-erg-2} for $ (x^{\ast}, y^{\ast}) $ yields
	\begin{equation}
		0
		\leq \Psi(\bar{x}_{K},y^{\ast}) - \Psi(x^{\ast},\bar{y}_{K})
		\leq \frac{1}{T_{K}} \left( \frac{1}{2 \tau} \normsq{x^{\ast}-x_{0}} + \frac{1}{2 \sigma} \normsq{y^{\ast}-y_{0}} \right),
	\end{equation}
	with
	\begin{equation}
		T_{K} = \sum_{k=0}^{K-1} t_{k},
			\hspace{5mm}
		\bar{x}_{K} = \frac{1}{T_{K}} \sum_{k=0}^{K-1} t_{k}x_{k+1},
			\hspace{5mm}
		\bar{y}_{K} = \frac{1}{T_{K}} \sum_{k=0}^{K-1} t_{k}y_{k+1}.
	\end{equation}
	Using~\eqref{tk-pdhg} to get $ t_{k} = 1 $ for all $ k \geq 0 $ in the above expressions gives
	\begin{equation}
		T_{K}
		= \sum_{k=0}^{K-1} t_{k}
		= K,
			\hspace{5mm}
		\bar{x}_{K} = \frac{1}{K} \sum_{k=0}^{K-1} x_{k+1},
			\hspace{5mm}
		\bar{y}_{K} = \frac{1}{K} \sum_{k=0}^{K-1} y_{k+1}.
	\end{equation}
	Finally we derive for all $ K \geq 1 $
	\begin{equation}
		0
		\leq \Psi(\bar{x}_{K},y^{\ast}) - \Psi(x^{\ast},\bar{y}_{K})
		\leq \frac{1}{K} \left( \frac{1}{2 \tau} \normsq{x^{\ast}-x_{0}} + \frac{1}{2 \sigma} \normsq{y^{\ast}-y_{0}} \right).
	\end{equation}
\end{proof}

\subsubsection{Convex-strongly concave setting}\label{convex-strconcave}
For the remainder of this section we assume that the function $ g $ is convex with modulus $ \nu > 0 $, meaning it is $ \nu $-strongly convex. In this case the choice~\eqref{stepsize-pdhg} leads to adaptive parameters and accelerated convergence.
\begin{proposition}\label{cor:stepsize-convex-strconcave}
	Let $ c_{\alpha} > L_{yx} \geq 0 $, $ \theta_{0} = 1 $ and $ \tau_{0}, \, \sigma_{0} > 0 $ such that
	\begin{equation}
		\left( c_{\alpha} L_{yx} \tau_{0} + 2 L_{yy} \right) \sigma_{0} < 1.
	\end{equation}
	If $ \nu > 0 $ then $ (\tau_{k})_{k \geq 0} $, $ (\sigma_{k})_{k \geq 0} $ and $ (\theta_{k})_{k \geq 0} $ as defined in Proposition~\ref{prop:stepsize} are adaptive, in particular we have
	\begin{equation}
		\label{adapt-param}
		\theta_{k+1} = \frac{1}{\sqrt{1 + \nu \sigma_{k}}} < 1,
			\hspace{5mm}
		\tau_{k+1} = \frac{\tau_{k}}{\theta_{k+1}} > \tau_{k},
			\hspace{5mm}
		\sigma_{k+1} = \theta_{k+1} \sigma_{k} < \sigma_{k}
		\quad
		\text{for all } k \geq 0.
	\end{equation}
\end{proposition}

\begin{proof}
The statements follow directly from Proposition~\ref{prop:stepsize} for $ \nu > 0 $.
\end{proof}

To obtain statements regarding the (accelerated) convergence rates in the convex-strongly~concave setting, we look at the behaviour of the sequences of step size parameters $(\tau_{k})_{k \geq 0} $ and $(\sigma_{k})_{k \geq 0}$ for $ k \to + \infty $.
\begin{proposition}
	Let $ \theta_{0} = 1 $, $ \tau_{0} > 0 $,
	\begin{equation}
		0 < \sigma_{0} \leq \frac{9 + 3 \sqrt{13}}{2 \nu},
	\end{equation}
	and for all $ k \geq 0 $ denote
	\begin{equation}
		\gamma_{k} := \frac{\tau_{k}}{\sigma_{k}}.
	\end{equation}
	Then with the choice of adaptive parameters~\eqref{adapt-param} we have for all $ k \geq 0 $
	\begin{equation}
		\gamma_{k} \geq \frac{\nu^{2} \tau_{0} \sigma_{0}}{9} k^{2} \quad \mbox{and} \quad
		\tau_{k} \geq \frac{\nu \tau_{0} \sigma_{0}}{3} k,
	\end{equation}
	and for all $ k \geq 1 $
	\begin{equation}
		\sigma_{k} \leq \frac{3}{\nu} \frac{1}{k}.
	\end{equation}
\end{proposition}

\begin{proof}
	By~\eqref{stepsize-pdhg} we conclude that for all $k \geq 0$
	\begin{equation}
		\gamma_{k+1} = \gamma_{k} (1 + \nu \sigma_{k}),
	\end{equation}
	and further
	\begin{equation}
		\sigma_{k+1} = \sigma_{k} \sqrt{ \frac{\gamma_{k}}{\gamma_{k+1}} },
	\end{equation}
	which, applied recursively, gives
	\begin{equation}
		\sigma_{k}
			= \sigma_{0} \sqrt{ \frac{\gamma_{0}}{\gamma_{k}} }
			= \sqrt{ \tau_{0} \sigma_{0} } \frac{1}{\sqrt{\gamma_{k}}}.
	\end{equation}
	We obtain
	\begin{equation}
		\gamma_{k+1}
			= \gamma_{k} (1 + \nu \sigma_{k})
			= \gamma_{k} + \nu \sqrt{ \tau_{0} \sigma_{0} } \sqrt{\gamma_{k}},
	\end{equation}
	which we will use to show by induction that for all $k \geq 0$
	\begin{equation}
		\label{gamma}
		\gamma_{k} \geq \frac{ \nu^{2} \tau_{0} \sigma_{0} }{9} k^{2}.
	\end{equation}
For $ k = 0 $ the statement trivially holds, whereas for $ k = 1 $ we need to verify that
	\begin{equation}
		\gamma_{1}
			= \gamma_{0} + \nu \sqrt{ \tau_{0} \sigma_{0} } \sqrt{\gamma_{0}}
			= \frac{\tau_{0}}{\sigma_{0}} \left( 1 + \nu \sigma_{0} \right)
			\geq \frac{ \nu^{2} \tau_{0} \sigma_{0} }{9},
	\end{equation}
	which is equivalent to the following quadratic inequality
	\begin{equation}
		\sigma_{0}^{2} - \frac{9}{\nu} \sigma_{0} - \frac{9}{\nu^{2}} \leq 0,
	\end{equation}
	and guaranteed to hold by our initial choice of $ \sigma_{0} > 0 $.
	Now let $ k \geq 1 $ and assume that~\eqref{gamma} holds. Then
	\begin{equation}
		\gamma_{k+1}
			= \gamma_{k} + \nu \sqrt{ \tau_{0} \sigma_{0} } \sqrt{\gamma_{k}}
			\geq \frac{ \nu^{2} \tau_{0} \sigma_{0} }{9} k^{2} + \frac{ \nu^{2} \tau_{0} \sigma_{0} }{3} k
			\geq \frac{ \nu^{2} \tau_{0} \sigma_{0} }{9} (k+1)^{2}.
	\end{equation}
This shows the validity of~\eqref{gamma} for all $ k \geq 0 $.

Now we can use inequality~\eqref{gamma} to deduce the convergence behaviour of the sequences $(\tau_{k})_{k \geq 0}$ and $(\sigma_{k})_{k \geq 0}$ for $ k \to + \infty $. We get for all $ k \geq 0 $
	\begin{equation}
		\label{tau}
		\tau_{k}
		= \sigma_{k} \gamma_{k}
		= \sqrt{ \tau_{0} \sigma_{0} } \sqrt{\gamma_{k}}
		\geq \frac{ \nu \tau_{0} \sigma_{0} }{3} k,
	\end{equation}
	which, combined with
	\begin{equation}
		\label{tau-sigma}
		\tau_{k} \sigma_{k}
		= \frac{\tau_{k}^{2}}{\gamma_{k}}
		= \tau_{0} \sigma_{0},
	\end{equation}
	gives for all $ k \geq 1 $
	\begin{equation}
		\sigma_{k} \leq \frac{3}{\nu} \frac{1}{k}.
	\end{equation}
\end{proof}

Now we are ready to prove the convergence results in the convex-strongly concave setting.
\begin{theorem}\label{convex-str-concave}
	Let $ c_{\alpha} > L_{yx} \geq 0 $, $ \theta_{0} = 1 $ and $ \tau_{0}, \, \sigma_{0} > 0 $ such that
	\begin{equation}
		\left( c_{\alpha} L_{yx} \tau_{0} + 2 L_{yy} \right) \sigma_{0} < 1 \quad \mbox{and} \quad
		0 < \sigma_{0} \leq \frac{9 + 3 \sqrt{13}}{2 \nu}.
	\end{equation}
	Let $ (x^{\ast}, y^{\ast}) \in \cH \times \cG $ be a saddle point of~\eqref{min-max}. Then for $ (x_{k}, y_{k})_{k \geq 0} $ being the sequence generated by OGAProx with the choice of adaptive parameters
	\begin{equation}
		\theta_{k+1} = \frac{1}{\sqrt{1 + \nu \sigma_{k}}} < 1,
			\hspace{5mm}
		\tau_{k+1} = \frac{\tau_{k}}{\theta_{k+1}} > \tau_{k},
			\hspace{5mm}
		\sigma_{k+1} = \theta_{k+1} \sigma_{k} < \sigma_{k}
		\quad
		\text{for all } k \geq 0,
	\end{equation}
	we have for all $ K \geq 1 $
	\begin{equation}
		\norm{y^{\ast}-y_{K}}
			\leq \frac{c_{1}}{K} {\left( \frac{1}{2 \tau_{0}} \normsq{x^{\ast}-x_{0}} + \frac{1}{2 \sigma_{0}} \normsq{y^{\ast}-y_{0}} \right)}^{\frac{1}{2}},
	\end{equation}
	with $ c_{1} := \sqrt{\frac{18}{\nu^{2} \sigma_{0} \delta}} $, where $ \delta > 0 $ is defined in~\eqref{delta}.
	Furthermore, for $K \geq 1$,  denote
	\begin{equation}
		T_{K} = \sum_{k=0}^{K-1} t_{k},
		\hspace{5mm}
		\bar{x}_{K} = \frac{1}{T_{K}} \sum_{k=0}^{K-1} t_{k} x_{k+1},
		\hspace{5mm}
		\bar{y}_{K} = \frac{1}{T_{K}} \sum_{k=0}^{K-1} t_{k} y_{k+1},
	\end{equation}
	where $t_k = \frac{\tau_k}{\tau_0}$ for all $k \geq 0$ (see also \eqref{tk-pdhg}).
	Then for all $ K \geq 2 $ it holds
	\begin{equation}
		0
		\leq \Psi(\bar{x}_{K},y^{\ast}) - \Psi(x^{\ast},\bar{y}_{K})
		\leq \frac{c_{2}}{K^{2}} \left( \frac{1}{2 \tau_{0}} \normsq{x^{\ast}-x_{0}} + \frac{1}{2 \sigma_{0}} \normsq{y^{\ast}-y_{0}} \right),
	\end{equation}
	with $ c_{2} := \frac{12}{\nu \sigma_{0}} $.
\end{theorem}

\begin{proof}
	Let $ K \geq 1 $ and let $ (x^{\ast},y^{\ast}) \in \cH \times \cG $ be an arbitrary but fixed saddle point. First we will prove the convergence rate of the sequence of iterates $ (y_{k})_{k \geq 0} $. Plugging the particular choice of parameters~\eqref{adapt-param} into~\eqref{est-gap-erg-1} for $ (x^{\ast}, y^{\ast}) $, we obtain
	\begin{equation}
		\begin{split}
			\frac{1}{2 \tau_{0}} \normsq{x^{\ast}-x_{0}} + \frac{1}{2 \sigma_{0}} \normsq{y^{\ast}-y_{0}} & \geq \frac{1}{2 \tau_{0}} \normsq{x^{\ast}-x_{K}}
			+ \frac{\tau_{K}}{\sigma_{K}} \left( 1 - \sigma_{K} \theta_{K} ( L_{yx} \alpha_{K} + L_{yy} ) \right) \frac{1}{2 \tau_{0}} \normsq{y^{\ast}-y_{K}}\\
			&\geq \gamma_{K} \frac{\delta}{2 \tau_{0}} \normsq{y^{\ast}-y_{K}},
		\end{split}
	\end{equation}
	where we use~\eqref{stepsize-delta} in Assumption~\ref{A1} for the last inequality. Combining this with~\eqref{gamma} we derive
	\begin{equation}
		\norm{y^{\ast}-y_{K}}
		\leq \frac{c_{1}}{K} {\left( \frac{1}{2 \tau_{0}} \normsq{x^{\ast}-x_{0}} + \frac{1}{2 \sigma_{0}} \normsq{y^{\ast}-y_{0}} \right)}^{\frac{1}{2}},
	\end{equation}
	with $ c_{1} := \sqrt{\frac{18}{\nu^{2} \sigma_{0} \delta}} $.

	Next we will show the convergence rate of the minimax gap at the ergodic sequences. Writing~\eqref{est-gap-erg-2} for $ (x^{\ast}, y^{\ast}) $, we obtain
	\begin{equation}
		\label{est-gap-erg-3}
		0
		\leq \Psi(\bar{x}_{K},y^{\ast}) - \Psi(x^{\ast},\bar{y}_{K})
		\leq \frac{1}{T_{K}} \left( \frac{1}{2 \tau_{0}} \normsq{x^{\ast}-x_{0}} + \frac{1}{2 \sigma_{0}} \normsq{y^{\ast}-y_{0}} \right).
	\end{equation}
Plugging the particular choice of $ t_{k} = \frac{\tau_{k}}{\tau_{0}} $ for all $ k \geq 0 $  from~\eqref{tk-pdhg} into the definition of $ T_{K} $, together with~\eqref{tau} yields
	\begin{equation}
		T_{K}
		= \frac{1}{\tau_{0}} \sum_{k=0}^{K-1} \tau_{k}
		\geq \frac{\nu \sigma_{0}}{3} \sum_{k=0}^{K-1} k
		= \frac{\nu \sigma_{0}}{6} K(K-1).
	\end{equation}
	Combining this inequality with~\eqref{est-gap-erg-3}, we obtain for all $ K \geq 2 $
	\begin{equation}
		0
		\leq \Psi(\bar{x}_{K},y^{\ast}) - \Psi(x^{\ast},\bar{y}_{K})
		\leq \frac{c_{2}}{K^{2}} \left( \frac{1}{2 \tau_{0}} \normsq{x^{\ast}-x_{0}} + \frac{1}{2 \sigma_{0}} \normsq{y^{\ast}-y_{0}} \right),
	\end{equation}
	with $ c_{2} := \frac{12}{\nu \sigma_{0}} $, which concludes the proof.
\end{proof}

\section{Strongly convex-strongly concave setting}\label{sec:strconvex-strconcave}
For this section we assume that the function $ g $ is convex with modulus $ \nu > 0 $, meaning it is $ \nu $-strongly convex. In addition to the assumptions we had until now, for this section we also assume that for all $ y \in \dom g $ the function $ \Phi(\ph,y): \cH \to \R \cup \{ + \infty \} $ is $ \mu $-strongly convex with modulus $ \mu > 0 $. This means that the saddle function $(x,y) \mapsto \Psi(x, y) $ is strongly convex-strongly concave.

As in the previous section we will state two step size assumptions that will be needed for the convergence analysis. These again will be followed by preparatory observations and a result to guarantee the validity of the stated assumptions. The section will be closed with the formulation and proof of convergence results.

\begin{assumption}\label{A2}
	We assume that the step sizes $ \tau_{k} $, $ \sigma_{k} $ and the momentum parameter $ \theta_{k} $ are constant
	\begin{equation}
		\theta_{k} = \theta_{0} =: \theta,
			\hspace{5mm}
		\tau_{k} = \tau_{0} =: \tau,
			\hspace{5mm}
		\sigma_{k} = \sigma_{0} =: \sigma \quad \forall k \geq 0,
	\end{equation}
	and satisfy
	\begin{equation}
		\label{stepsize-lin-1}
		1 + \mu \tau = \frac{1}{\theta},
			\hspace{5mm}
		1 + \nu \sigma = \frac{1}{\theta},
	\end{equation}
	with
	\begin{equation}
		\label{theta-lin}
		0 < \theta < 1.
	\end{equation}
	Furthermore we assume that there exists $ \alpha > 0 $ such that
	\begin{equation}
		\label{stepsize-lin-2}
		\frac{L_{yx}}{\alpha} \leq \frac{1}{\tau},
			\hspace{5mm}
		L_{yy} \leq \frac{1 - \theta \sigma (\alpha L_{yx} + L_{yy})}{\sigma},
	\end{equation}
	with
	\begin{equation}
		\label{stepsize-lin-3}
		1 - \theta \sigma (\alpha L_{yx} + L_{yy}) > 0.
	\end{equation}
\end{assumption}

\subsection{Preliminary considerations}
We take an arbitrary $(x,y) \in \cH \times \cG $ and let $ k \geq 0 $. Following similar considerations along~\eqref{subdiff-g}-\eqref{subdiff-phi}, additionally taking into account the $ \mu $-strong convexity of $ \Phi (\ph, y) $ for $ y \in \dom g $, instead of~\eqref{comb} we derive
\begin{equation}
	\begin{split}
		\Psi(x_{k+1},y)-\Psi(x,y_{k+1}) \leq & \ \theta \sprod{q_{k}}{y_{k} - y} - \sprod{q_{k+1}}{y_{k+1} - y}\\
			& - \frac{\mu}{2} \normsq{x - x_{k+1}} - \frac{\nu}{2} \normsq{y - y_{k+1}} + \theta \sprod{q_{k}}{y_{k+1} - y_{k}}\\
			& + \frac{1}{2\tau} \left(-\normsq{x_{k} - x_{k+1}} - \normsq{x - x_{k+1}} + \normsq{x - x_{k}}\right)\\
			& + \frac{1}{2\sigma} \left(-\normsq{y_{k} - y_{k+1}} - \normsq{y - y_{k+1}} + \normsq{y - y_{k}}\right)\\
		\leq & \ \frac{1}{2 \tau} \normsq{x - x_{k}} + \frac{1}{2 \sigma} \normsq{y - y_{k}} + \theta \sprod{q_{k}}{y_{k} - y}\\
			&- \frac{1 + \mu \tau}{2 \tau} \normsq{x - x_{k+1}} - \frac{1 + \nu \sigma}{2 \sigma} \normsq{y - y_{k+1}} - \sprod{q_{k+1}}{y_{k+1} - y}\\
			&+ \frac{\theta L_{yx}}{2 \alpha} \normsq{x_{k} - x_{k-1}} - \frac{1}{2 \tau} \normsq{x_{k+1} - x_{k}}\\
			&+ \frac{\theta L_{yy}}{2} \normsq{y_{k} - y_{k-1}} - \frac{1 - \theta \sigma (\alpha L_{yx} + L_{yy})}{2 \sigma} \normsq{y_{k+1} - y_{k}}.
	\end{split}
\end{equation}
By~\eqref{stepsize-lin-1} in Assumption~\ref{A2} and for $\alpha >0$ fulfilling \eqref{stepsize-lin-2}-\eqref{stepsize-lin-3}, we obtain
\begin{equation}
	\begin{split}
		\Psi(x_{k+1},y)-\Psi(x,y_{k+1}) \leq & \ \frac{1}{2 \tau} \normsq{x - x_{k}} + \frac{1}{2 \sigma} \normsq{y - y_{k}} + \theta \sprod{q_{k}}{y_{k} - y}\\
			&- \frac{1}{2 \tau} \frac{1}{\theta} \normsq{x - x_{k+1}} - \frac{1}{2 \sigma} \frac{1}{\theta} \normsq{y - y_{k+1}} - \sprod{q_{k+1}}{y_{k+1} - y}\\
			&+ \frac{\theta L_{yx}}{2 \alpha} \normsq{x_{k} - x_{k-1}} - \frac{1}{2 \tau} \normsq{x_{k+1} - x_{k}}\\
			&+ \frac{\theta L_{yy}}{2} \normsq{y_{k} - y_{k-1}} - \frac{1 - \theta \sigma (\alpha L_{yx} + L_{yy})}{2 \sigma} \normsq{y_{k+1} - y_{k}},
	\end{split}
\end{equation}
which together with~\eqref{stepsize-lin-2} and~\eqref{stepsize-lin-3} gives
\begin{equation}
	\label{comb-lin}
	\begin{split}
		\Psi(x_{k+1},y)-\Psi(x,y_{k+1}) \leq & \ \frac{1}{2 \tau} \left( \normsq{x - x_{k}} - \frac{1}{\theta} \normsq{x - x_{k+1}}\right) + \frac{1}{2 \sigma} \left( \normsq{y - y_{k}} - \frac{1}{\theta} \normsq{y - y_{k+1}} \right)\\
			& + \theta \sprod{q_{k}}{y_{k} - y} - \sprod{q_{k+1}}{y_{k+1} - y} + \frac{1}{2 \tau} \left( \theta \normsq{x_{k} - x_{k-1}} - \normsq{x_{k+1} - x_{k}} \right)\\
			& + \frac{1}{2 \tilde{\sigma}} \left( \theta \normsq{y_{k} - y_{k-1}} - \normsq{y_{k+1} - y_{k}} \right),
	\end{split}
\end{equation}
where
\begin{equation}\label{cuc}
	\tilde{\sigma} := \frac{\sigma}{1 - \theta \sigma (\alpha L_{yx} + L_{yy})}.
\end{equation}

Let $ K \geq 1 $ and as in~\eqref{erg-sequ} denote
\begin{equation}
	T_{K} = \sum_{k=0}^{K-1} t_{k},
	\hspace{5mm}
	\bar{x}_{K} = \frac{1}{T_{K}} \sum_{k=0}^{K-1} t_{k}x_{k+1},
	\hspace{5mm}
	\bar{y}_{K} = \frac{1}{T_{K}} \sum_{k=0}^{K-1} t_{k}y_{k+1}.
\end{equation}
with $ t_{k} > 0 $ defined as in~\eqref{tk}, in other words
\begin{equation}
	t_{k} = \theta^{-k} \quad \forall k \geq 0.
\end{equation}

Multiplying both sides of~\eqref{comb-lin} by $ t_{k} > 0 $ yields
\begin{equation}
	\begin{split}
		\frac{1}{\theta^{k}} \left(\Psi(x_{k+1},y)-\Psi(x,y_{k+1})\right) &\leq \frac{1}{2 \tau} \left( \frac{1}{\theta^{k}} \normsq{x - x_{k}} - \frac{1}{\theta^{k+1}} \normsq{x - x_{k+1}}\right)\\
			&\hspace{5mm}+ \frac{1}{2 \sigma} \left( \frac{1}{\theta^{k}} \normsq{y - y_{k}} - \frac{1}{\theta^{k+1}} \normsq{y - y_{k+1}} \right)\\
			&\hspace{5mm}+ \frac{1}{\theta^{k-1}}\sprod{q_{k}}{y_{k} - y} - \frac{1}{\theta^{k}} \sprod{q_{k+1}}{y_{k+1} - y}\\
			&\hspace{5mm}+ \frac{1}{2 \tau} \left( \frac{1}{\theta^{k-1}} \normsq{x_{k} - x_{k-1}} - \frac{1}{\theta^{k}} \normsq{x_{k+1} - x_{k}} \right)\\
			&\hspace{5mm}+ \frac{1}{2 \tilde{\sigma}} \left( \frac{1}{\theta^{k-1}} \normsq{y_{k} - y_{k-1}} - \frac{1}{\theta^{k}} \normsq{y_{k+1} - y_{k}} \right).
	\end{split}
\end{equation}
Summing up the above inequality for $ k = 0, \ldots, K-1 $ and taking into account Jensen's inequality for the convex function $ \Psi(\ph, y) - \Psi(x, \ph) $ give
\begin{equation}
	\begin{split}
		T_{K} \left( \Psi(\bar{x}_{K},y)-\Psi(x,\bar{y}_{K}) \right) & \leq \sum_{k=0}^{K-1} \frac{1}{\theta^{k}} \left( \Psi(x_{k+1},y)-\Psi(x,y_{k+1}) \right)\\
			&\leq \frac{1}{2 \tau} \left( \normsq{x - x_{0}} - \frac{1}{\theta^{K}} \normsq{x - x_{K}} \right)
			+ \frac{1}{2 \sigma} \left( \normsq{y - y_{0}} - \frac{1}{\theta^{K}} \normsq{y - y_{K}} \right)\\
				&\hspace{5mm}- \frac{1}{\theta^{K-1}} \sprod{q_{K}}{y_{K} - y}
				- \frac{1}{\theta^{K-1}} \frac{1}{2 \tau} \normsq{x_{K} - x_{K-1}}
				- \frac{1}{\theta^{K-1}} \frac{1}{2 \tilde{\sigma}} \normsq{y_{K} - y_{K-1}}\\
			&\leq \frac{1}{2 \tau} \left( \normsq{x - x_{0}} - \frac{1}{\theta^{K}} \normsq{x - x_{K}} \right)
			+ \frac{1}{2 \sigma} \left( \normsq{y - y_{0}} - \frac{1}{\theta^{K}} \normsq{y - y_{K}} \right)\\
				&\hspace{5mm}+ \frac{1}{\theta^{K-1}} \frac{L_{yx}}{2} \left( \frac{1}{\alpha} \normsq{x_{K} - x_{K-1}} + \alpha \normsq{y_{K} - y} \right)
				- \frac{1}{\theta^{K-1}} \frac{1}{2 \tau} \normsq{x_{K} - x_{K-1}}\\
				&\hspace{5mm}+ \frac{1}{\theta^{K-1}} \frac{L_{yy}}{2} \left( \normsq{y_{K} - y_{K-1}} + \normsq{y_{K} - y} \right)
				- \frac{1}{\theta^{K-1}} \frac{1}{2 \tilde{\sigma}} \normsq{y_{K} - y_{K-1}}\\
			&= \frac{1}{2 \tau} \normsq{x - x_{0}} + \frac{1}{2 \sigma} \normsq{y - y_{0}}\\
				&\hspace{5mm}- \frac{1}{\theta^{K}} \frac{1}{2 \tau} \normsq{x - x_{K}}
				- \frac{1}{\theta^{K}} \frac{1 - \theta \sigma (\alpha L_{yx} + L_{yy})}{2 \sigma} \normsq{y - y_{K}}\\
				&\hspace{5mm}- \frac{1}{2 \theta^{K-1}} \left( \frac{1}{\tau} - \frac{L_{yx}}{\alpha} \right) \normsq{x_{K} - x_{K-1}}
				- \frac{1}{2 \theta^{K-1}} \left( \frac{1}{\tilde{\sigma}} - L_{yy} \right) \normsq{y_{K} - y_{K-1}},
	\end{split}
\end{equation}
where in the second inequality we use \eqref{estim}.
Omitting the last two terms which are non positive by~\eqref{stepsize-lin-2}, we obtain for all $ K \geq 1 $
\begin{equation}
	\label{lin-conv}
	\begin{split}
		T_{K} \theta^{K} \big( \Psi(\bar{x}_{K},y)-\Psi(x,\bar{y}_{K}) \big) + & \frac{1}{2 \tau} \normsq{x - x_{K}} + \frac{1}{2 \tilde{\sigma}} \normsq{y - y_{K}} \leq \theta^{K} \left( \frac{1}{2 \tau} \normsq{x - x_{0}} + \frac{1}{2 \sigma} \normsq{y - y_{0}} \right),
	\end{split}
\end{equation}
which we will use to obtain our convergence results in the following.

\subsection{Fulfilment of step size assumptions}
In this subsection we will investigate a particular choice of parameters $ \tau $, $ \sigma $ and $ \theta $ such that Assumption~\ref{A2} holds.

\begin{proposition}\label{prop:stepsize-strconvex-strconcave}
For $\alpha >0$ define
	\begin{equation}
		\label{theta-tilde}
		\tilde{\theta} :=
		\max \left\{
			\frac{L_{yx}}{\alpha \mu + L_{yx}}, \,
			\frac{\alpha L_{yx} + 2 L_{yy}}{\nu + \alpha L_{yx} + 2L_{yy}} \right\}.
	\end{equation}
Let $ \theta > 0 $ such that
	\begin{equation}
		\label{theta}
		0 \leq \tilde{\theta} < \theta < 1,
	\end{equation}
	and set
	\begin{equation}
		\label{tau-sigma-lin-gap}
		\tau = \frac{1}{\mu} \frac{1 - \theta}{\theta} \quad \mbox{and} \quad
		\sigma = \frac{1}{\nu} \frac{1 - \theta}{\theta}.
	\end{equation}
	Then $ \tau $, $ \sigma $ and $ \theta $ fulfil Assumption~\ref{A2}.
\end{proposition}

\begin{proof}
If $L_{yx} = L_{yy} =0$, then the conclusion follows immediately. Assume that $L_{yx} + L_{yy} >0$. It is easy to verify that definition~\eqref{theta-tilde} yields
	\begin{equation}
		0 < \tilde{\theta} < 1
	\end{equation}
	and that~\eqref{tau-sigma-lin-gap} is equivalent to~\eqref{stepsize-lin-1} where~\eqref{theta-lin} is ensured by~\eqref{theta}. Furthermore, plugging the specific form of the step sizes~\eqref{tau-sigma-lin-gap} into~\eqref{stepsize-lin-2} we obtain for the first inequality of~\eqref{stepsize-lin-2}
	\begin{equation}
		\frac{L_{yx}}{\alpha} \leq \frac{\mu \theta}{1 - \theta},
	\end{equation}
	which is equivalent to
	\begin{equation}
		\theta \geq \frac{L_{yx}}{\alpha \mu + L_{yx}}.
	\end{equation}
	Note that by~\eqref{theta} we have
	\begin{equation}
		0 \leq \frac{L_{yx}}{\alpha \mu + L_{yx}} \leq \tilde{\theta} < \theta < 1.
	\end{equation}
	Similarly, the second inequality of~\eqref{stepsize-lin-2} is equivalent to the following quadratic inequality
	\begin{equation}
		\theta^{2} - \frac{\alpha L_{yx} - \nu}{\alpha L_{yx} + L_{yy}} \theta - \frac{L_{yy}}{\alpha L_{yx} + L_{yy}} \geq 0.
	\end{equation}
The non negative solution of the associated quadratic equation reads
	\begin{equation}
\rho := \frac{1}{2}	\left( \frac{\alpha L_{yx} - \nu}{\alpha L_{yx} + L_{yy}} + \sqrt{ \left( \frac{\alpha L_{yx} - \nu}{\alpha L_{yx} + L_{yy}} \right)^{2} + \frac{4 L_{yy}}{\alpha L_{yx} + L_{yy}} } \right)
		\geq 0.
	\end{equation}
Since
	\begin{equation}
		0 \leq \rho < \frac{\alpha L_{yx} + 2L_{yy}}{\nu + \alpha L_{yx} + 2L_{yy}}  \leq \tilde{\theta} < \theta < 1,
	\end{equation}
the second inequality of~\eqref{stepsize-lin-2} is also fulfilled. In order to see that
\begin{equation*}
	\rho < \frac{\alpha L_{yx} + 2L_{yy}}{\nu + \alpha L_{yx} + 2L_{yy}},
\end{equation*}
we notice that this inequality is equivalent to
\begin{equation}
	\left( \frac{\alpha L_{yx} - \nu}{\alpha L_{yx} + L_{yy}} \right)^{2} + \frac{4 L_{yy}}{\alpha L_{yx} + L_{yy}}
	<
	\frac{\big( \nu^{2} + 2 \nu L_{yy} + (\alpha L_{yx} + 2 L_{yy})^{2} \big)^{2}}{(\nu + \alpha L_{yx} + 2L_{yy})^{2}(\alpha L_{yx} + L_{yy})^{2}},
\end{equation}
which holds if and only if
\begin{align*}
		 & \big( (\alpha L_{yx} - \nu)^{2} + 4 L_{yy}(\alpha L_{yx} + L_{yy}) \big)
			(\nu + \alpha L_{yx} + 2L_{yy})^{2}\\
		< & \ (\nu^{2} + 2 \nu L_{yy})^{2}
			+ 2(\nu^{2} + 2 \nu L_{yy})(\alpha L_{yx} + 2L_{yy})^{2}
			+ (\alpha L_{yx} + 2 L_{yy})^{4}
\end{align*}
or, equivalently,
\begin{equation*}
	0< 4 \nu^{2} (\alpha L_{yx} + L_{yy})^{2}.
\end{equation*}
For the remaining condition~\eqref{stepsize-lin-3} to hold we need to ensure
	\begin{equation}
		\theta > \frac{\alpha L_{yx} + L_{yy} - \nu}{\alpha L_{yx} + L_{yy}}.
	\end{equation}
For this we observe that
\begin{equation}
	\begin{split}
		\rho &\geq \frac{1}{2}
			\left(
				\frac{\alpha L_{yx} - \nu}{\alpha L_{yx} + L_{yy}}
				+ \sqrt{ \left( \frac{\alpha L_{yx} + 2L_{yy} - \nu}{\alpha L_{yx} + L_{yy}} \right)^{2} }
			\right)\\
		&\geq \frac{1}{2}
			\frac{\alpha L_{yx} - \nu + \alpha L_{yx} + 2L_{yy} - \nu}{\alpha L_{yx} + L_{yy}} = \frac{\alpha L_{yx} + L_{yy} - \nu}{\alpha L_{yx} + L_{yy}}.
	\end{split}
\end{equation}
In conclusion, we obtain the following chain of inequalities
	\begin{equation}
		\frac{\alpha L_{yx} + L_{yy} - \nu}{\alpha L_{yx} + L_{yy}} \leq \rho
		< \frac{\alpha L_{yx} + 2L_{yy}}{\nu + \alpha L_{yx} + 2L_{yy}}
		\leq \tilde{\theta} < \theta < 1,
	\end{equation}
	which  is satisfied by~\eqref{theta}.
\end{proof}

\subsection{Convergence results}\label{strconvex-strconcave}
Now we can combine the previous results and prove the convergence statements in the strongly convex-strongly concave setting.
\begin{theorem}\label{str-convex-str-concave}
	Let $ (x^{\ast}, y^{\ast}) \in \cH \times \cG $ be a saddle point of~\eqref{min-max}. Then for $ (x_{k}, y_{k})_{k \geq 0} $ being the sequence generated by OGAProx with the choice of parameters
	\begin{equation}
		\tau = \frac{1}{\mu} \frac{1 - \theta}{\theta},
		\hspace{5mm}
		\sigma = \frac{1}{\nu} \frac{1 - \theta}{\theta},
		\hspace{5mm}
		0 \leq \tilde{\theta} < \theta < 1,
	\end{equation}
	with
	\begin{equation}
		\tilde{\theta} =
			\max \left\{
				\frac{L_{yx}}{\alpha \mu + L_{yx}}, \,
				\frac{\alpha L_{yx} + 2 L_{yy}}{\nu + \alpha L_{yx} + 2L_{yy}} \right\},
	\end{equation}
for $\alpha >0$, we denote for $ K \geq 1 $
	\begin{equation}
		T_{K} = \sum_{k=0}^{K-1} \theta^{-k},
			\hspace{5mm}
		\bar{x}_{K} = \frac{1}{T_{K}} \sum_{k=0}^{K-1} \theta^{-k} x_{k+1},
			\hspace{5mm}
		\bar{y}_{K} = \frac{1}{T_{K}} \sum_{k=0}^{K-1} \theta^{-k} y_{k+1},
	\end{equation}
	for which the following holds
	\begin{equation*}
			0 \leq \theta \big( \Psi(\bar{x}_{K},y^{\ast}) - \Psi(x^{\ast},\bar{y}_{K}) \big) +  \frac{1}{2 \tau} \normsq{x^{\ast} - x_{K}} + \frac{1}{2 \tilde{\sigma}} \normsq{y^{\ast} - y_{K}} \leq \theta^{K} \left( \frac{1}{2 \tau} \normsq{x^{\ast} - x_{0}} + \frac{1}{2 \sigma} \normsq{y^{\ast} - y_{0}} \right),
	\end{equation*}
	where $\tilde{\sigma} := \frac{\sigma}{1 - \theta \sigma (\alpha L_{yx} + L_{yy})}$.
\end{theorem}

\begin{proof}
	Let $ K \geq 1 $ and $ (x^{\ast}, y^{\ast}) \in \cH \times \cG $ be an arbitrary but fixed saddle point of~\eqref{min-max}. Writing~\eqref{lin-conv} for $ (x^{\ast}, y^{\ast}) $ we get
	\begin{align*}
			0 \leq & \ T_{K} \theta^{K} \big( \Psi(\bar{x}_{K},y^{\ast})-\Psi(x^{\ast},\bar{y}_{K}) \big) + \frac{1}{2 \tau} \normsq{x^{\ast} - x_{K}} + \frac{1}{2 \tilde{\sigma}} \normsq{y^{\ast} - y_{K}} \\
\leq & \ \theta^{K} \left( \frac{1}{2 \tau} \normsq{x^{\ast} - x_{0}} + \frac{1}{2 \sigma} \normsq{y^{\ast} - y_{0}} \right).
	\end{align*}
	Using
	\begin{equation}
		T_{K} = \sum_{k=0}^{K-1} \frac{1}{\theta^{k}} = \frac{1}{\theta^{K-1}} \frac{1 - \theta^{K}}{1 - \theta} \geq \frac{1}{\theta^{K-1}},
	\end{equation}
	finally we obtain for all $ K \geq 1 $
	\begin{equation*}
			0 \leq \theta \big( \Psi(\bar{x}_{K},y^{\ast})-\Psi(x^{\ast},\bar{y}_{K}) \big) +  \frac{1}{2 \tau} \normsq{x^{\ast} - x_{K}} + \frac{1}{2 \tilde{\sigma}} \normsq{y^{\ast} - y_{K}} \leq \theta^{K} \left( \frac{1}{2 \tau} \normsq{x^{\ast} - x_{0}} + \frac{1}{2 \sigma} \normsq{y^{\ast} - y_{0}} \right),
	\end{equation*}
	with $ 0 < \theta < 1 $ as defined in~\eqref{theta}.
\end{proof}

\section{Numerical experiments}\label{sec:numerics}
In this section we will treat three numerical applications of our method. The first one is of rather simple structure and has the purpose to highlight the convergence rates we obtained in the previous sections. The second one concerns multi kernel support vector machines to validate OGAProx on a more relevant application in practice, even though there are no theoretical guarantees for the ``metric'' reported there. The third numerical application addresses a classification problem incorporating minimax group fairness, which traces back to the solving of a minimax problem with nonsmooth coupling function.

\subsection{Nonsmooth-linear problem}
The first application we treat is to showcase the convergence rates we obtained in the previous sections and make a simple proof of concept. We look at the following nonsmooth-linear saddle point problem
\begin{equation}
	\label{nslSP}
	\min_{x \in \R^{d}} \, \max_{y \in \R^{n}} \Psi (x, y) :=
		\sprod{[x]_{+}}{Ay}
		- \left( \delta_{C}(y) + \frac{\nu}{2} \normsq{y} \right),
\end{equation}
with $ \nu \geq 0 $ and $ A \in \R^{d \times n} $, $ [ \ph ]_{+} $ being the component-wise positive part,
\begin{equation}
	[x]_{+} = {\big( \max \{ 0, x_{i} \} \big)}_{i=1}^{d},
\end{equation}
and $ C $ being the following convex polytope
\begin{equation}
	C := \{ y \in \R^{n} \ | \ A y \geqq 0 \}.
\end{equation}
For $ u = (u_{i})_{i=1}^{d} $, $ v = (v_{i})_{i=1}^{d} \in \R^{d} $ the relation $ u \geqq v $ denotes component-wise inequalities, namely,
\begin{equation}
	u \geqq v \: \Leftrightarrow \: u_{i} \geq v_{i} \quad \text{for } 1 \leq i \leq d.
\end{equation}

Then $ g: \R^{n} \to \R \cup \{ + \infty \} $ with
\begin{equation}
	g(y) := \delta_{C}(y) + \frac{\nu}{2} \normsq{y}
\end{equation}
is proper, lower semicontinuous and convex with modulus $ \nu \geq 0 $ and $ \dom g = C $. Moreover, $ \Phi : \R^{d} \times \R^{n} \to \R $ with
\begin{equation}
	\Phi (x, y) = \sum_{i = 1}^{d} \max \{ 0, x_{i} \} (Ay)_{i}
\end{equation}
has full domain, for all $ x \in \R^{d} $ we have that $ \Phi(x, \ph) $ is linear and for all $ y \in \dom g = C$ the function $ \Phi (\ph, y) $ is convex and continuous.

Furthermore, we obtain for all $ (x, y), \, (x', y') \in \R^{d} \times \dom g $
\begin{equation}
	\norm{\ny \Phi (x, y) - \ny \Phi (x', y')}
	= \norm{A^{T} \left( [x]_{+} - [x']_{+} \right)}
	\leq \norm{A} \norm{x - x'},
\end{equation}
hence~\eqref{cond-Phi_y} holds with $ L_{yx} = \norm{A} $ and $ L_{yy} = 0 $.

The algorithm~\eqref{alg_y}-\eqref{alg_x} iterates for $ k \geq 0 $
\begin{equation*}
		\left\{
		\begin{array}{rcl}
			v_{k} &=& y_{k} + \sigma_{k} \left[ (1 + \theta_{k}) \ny \Phi(x_{k}, y_{k}) - \theta_{k} \ny \Phi(x_{k-1}, y_{k-1})\right] = y_{k} + \sigma_{k} A^{T} \big( (1 + \theta_{k}) [x_{k}]_{+} - \theta_{k} [x_{k-1}]_{+} \big),\\
			y_{k+1}  &=& \prox{\sigma_{k} g}{v_{k}} = P_{C} \left( \frac{1}{1 + \nu \sigma_{k}} v_{k} \right),\\
			x_{k+1}  &=& \prox{\tau_{k} \Phi(\ph, y_{k+1})}{x_{k}},
		\end{array}
		\right.
\end{equation*}
where the calculation of the orthogonal projection on the set $C$ is a simple quadratic program and
\begin{equation}
	\prox{\tau \Phi (\ph, y)}{x} = \left( \prox{\tau (Ay)_{i} \max \{ 0, \ph \} }{x_{i}} \right)_{i = 1}^{d},
\end{equation}
where, for $i=1, ..., d$,
\begin{equation}
	\prox{\tau (Ay)_{i} \max \{ 0, \ph \} }{x_{i}}
	= \begin{cases}
		x_{i} & \text{if } x_{i} \leq 0, \vspace{1mm}\\
		0 & \text{if } 0 < x_{i} \leq \tau (Ay)_{i}, \vspace{1mm}\\
		x_{i} - \tau (Ay)_{i} & \text{if } x_{i} > \tau (Ay)_{i}.
	\end{cases}
\end{equation}

By writing the first order optimality conditions and using Lagrange duality we obtain the following characterisation.
\begin{equation}
	\begin{split}
		(x^{\ast}, y^{\ast}) \text{ is a saddle point of~\eqref{nslSP}}
			& \Leftrightarrow
				\left\{
					\begin{array}{l}
						0 \in \partial
							\left(
								\sprod{[ \ph ]_{+}}{Ay^{\ast}} -  \delta_{C}(y^{\ast}) - \frac{\nu}{2} \normsq{y^{\ast}}
							\right) (x^{\ast}) \vspace{1mm}\\
						0 \in \partial
							\left(
								\minus \sprod{A^{T}[x^{\ast}]_{+}}{\ph} + \delta_{C}(\ph) + \frac{\nu}{2} \normsq{\ph}
							\right)(y^{\ast})
					\end{array}
				\right.\\
			& \Leftrightarrow
				\left\{
					\begin{array}{l}
						0 \in \sum_{i = 1}^{d} (Ay^{\ast})_{i} \, \partial \max \{ 0, \ph \} (x^{\ast}_{i})\vspace{1mm}\\
						A^{T}[x^{\ast}]_{+} - \nu y^{\ast} \in N_{C}(y^{\ast})
					\end{array}
				\right.\\
			& \Leftrightarrow
				\left\{
					\begin{array}{l}
						\forall i = 1, \ldots, d :
							\big( (Ay^{\ast})_{i} > 0 \text{ and } x^{\ast}_{i} \leq 0 \big)
							\text{ or }\\
							\hspace{5mm} \big( (Ay^{\ast})_{i} = 0 \text{ and } x^{\ast}_{i} \in \R \big)
							\vspace{1mm}\\
						\sprod{A^{T} [x^{\ast}]_{+} - \nu y^{\ast}}{y^{\ast}} = 0
						\vspace{1mm}\\
						\nu y^{\ast} - A^{T} [x^{\ast}]_{+} \in A^{T} \left( \R^{d}_{+} \right)
					\end{array}
				\right.\\
			& \Leftrightarrow
				\left\{
					\begin{array}{l}
						\forall i = 1, \ldots, d :
							\big( (Ay^{\ast})_{i} > 0 \text{ and } x^{\ast}_{i} \leq 0 \big)
							\text{ or }\\
							\hspace{5mm} \big( (Ay^{\ast})_{i} = 0 \text{ and } x^{\ast}_{i} \in \R \big)
							\vspace{1mm}\\
						\nu \normsq{y^{\ast}} =
							\sprod{A^{T} [x^{\ast}]_{+} }{y^{\ast}} =
							\sprod{[x^{\ast}]_{+}}{A y^{\ast}} = 0
						\vspace{1mm}\\
						\nu y^{\ast} \in A^{T} \left([x^{\ast}]_{+} + \R^{d}_{+} \right).
					\end{array}
				\right.
	\end{split}
\end{equation}

This means, that for $ \nu = 0 $ we obtain
\begin{equation}
	(x^{\ast}, y^{\ast}) \text{ is a saddle point of~\eqref{nslSP}}
		\Leftrightarrow
			\left\{
				\begin{array}{l}
					\forall i = 1, \ldots, d :
						\big( (Ay^{\ast})_{i} > 0 \text{ and } x^{\ast}_{i} \leq 0 \big)
						\text{ or }\\
						\hspace{5mm} \big( (Ay^{\ast})_{i} = 0 \text{ and } x^{\ast}_{i} \in \R \big)
						\vspace{1mm}\\
					0 \in A^{T} \left([x^{\ast}]_{+} + \R^{d}_{+} \right)
				\end{array},
			\right.
\end{equation}
whereas for $ \nu > 0 $
\begin{equation}
	(x^{\ast}, y^{\ast}) \text{ is a saddle point of~\eqref{nslSP}}
		\Leftrightarrow
			\left\{
				\begin{array}{l}
					y^{\ast} = 0\vspace{1mm}\\
					0 \in A^{T} \left([x^{\ast}]_{+} + \R^{d}_{+} \right)
				\end{array}.
			\right.
\end{equation}

If $ A \in \R^{d \times n} $ has full row rank the inclusion
\begin{equation}
	0 \in A^{T} \left([x^{\ast}]_{+} + \R^{d}_{+} \right)
\end{equation}
is equivalent to
\begin{equation}
	x^{\ast} \leqq 0.
\end{equation}

\begin{figure}[h]
	\centering
	\includegraphics[width=0.48\linewidth]{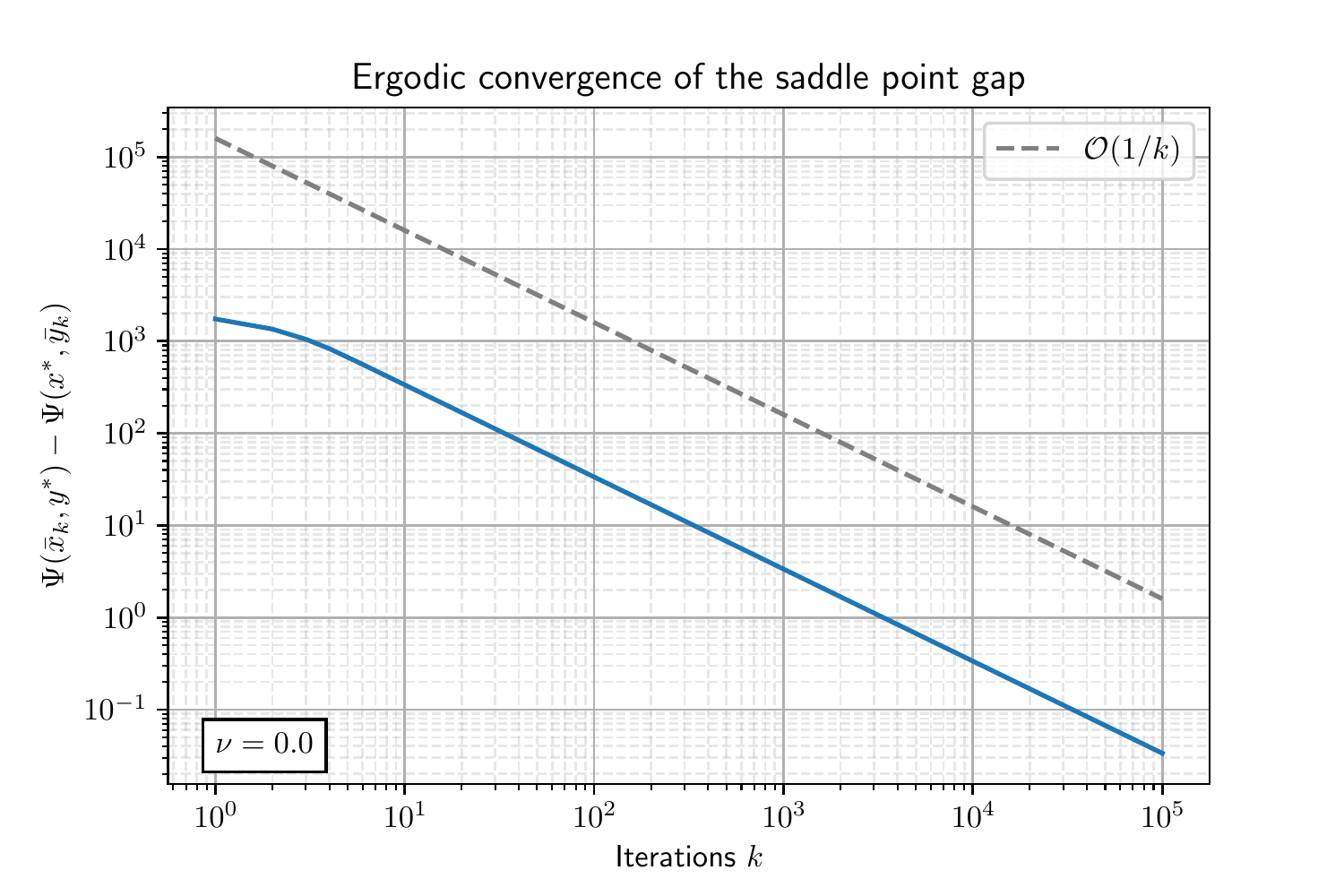}
	\caption{Convergence of the minimax gap like $ \mathcal{O} (\frac{1}{K}) $ for $ \nu = 0 $.}
	\label{fig:toy_nu=0}
\end{figure}

For the experiments we choose dimensions $ d = 250 $ and $ n = 350 $. For easier validation of the solution $ x^{\ast} $ we ensure that the matrix $ A \in \R^{d \times n} $ with entries drawn from a uniform distribution on the interval $ [ \minus 3, 3 ] $ has full row rank. The starting points $ x_{0} = x_{\minus 1} \in \R^{d} $ and $ y_{0} \in \R^{n} $ have entries drawn from a uniform distribution on the interval $ [ \minus 5, 5] $.

In the case $ \nu = 0 $, i.e., the regulariser $ g $ being merely convex, we proved weak asymptotic convergence of the iterates to some saddle point $ (x^{\ast}, y^{\ast}) $ and convergence of the minimax gap at the ergodic sequences to zero like $ \mathcal{O} (\frac{1}{K}) $ for any saddle point. The latter is illustrated in Figure~\ref{fig:toy_nu=0} for $ (x^{\ast}, y^{\ast}) \in \R^{d} \times \R^{n} $ with $ x^{\ast} \leqq 0 $ and $ y^{\ast} \in C  $ with $ y^{\ast} \neq 0 $ for a single random initialisation.

\begin{figure}[h]
	\centering
	\includegraphics[width=0.48\linewidth]{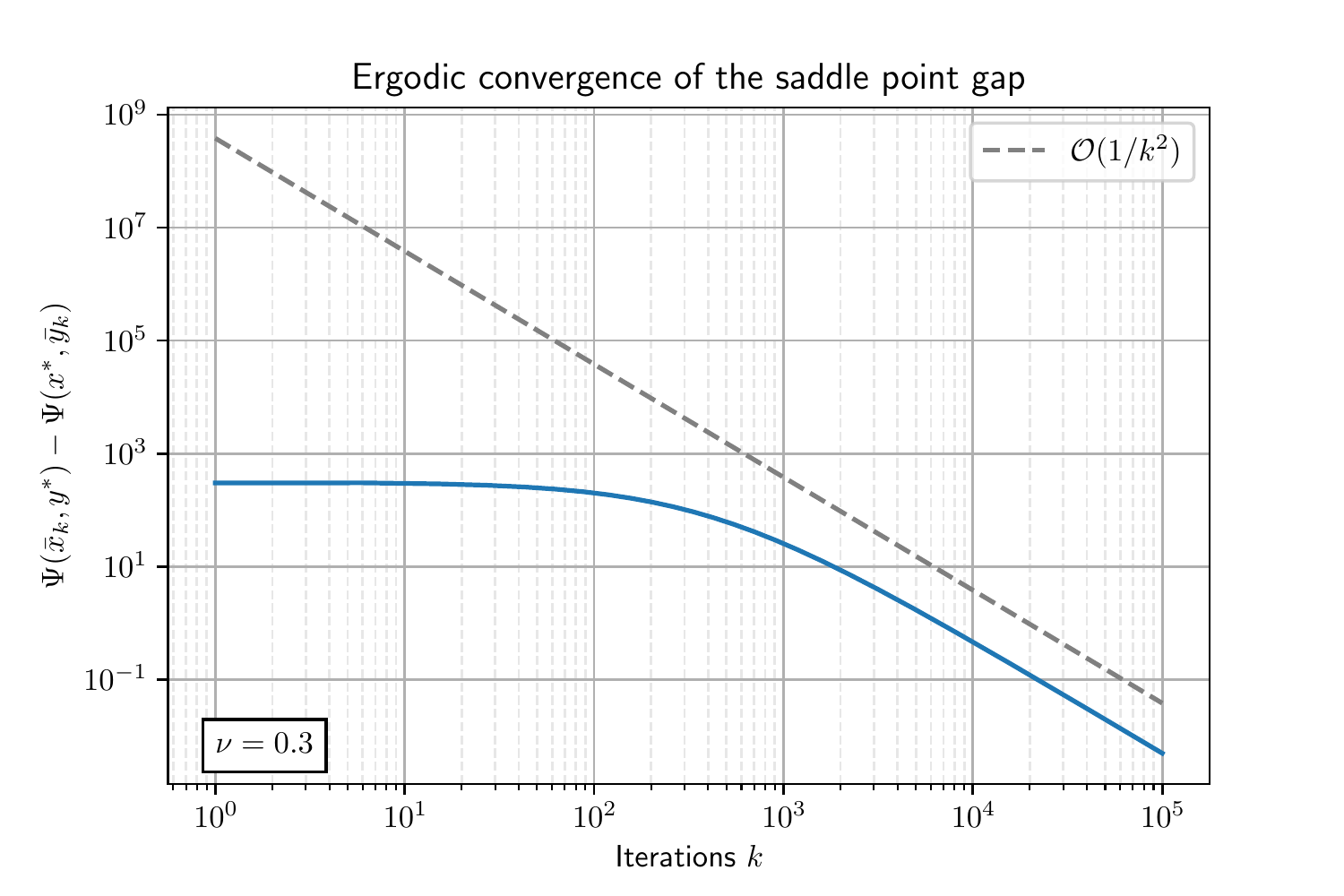}
	\includegraphics[width=0.48\linewidth]{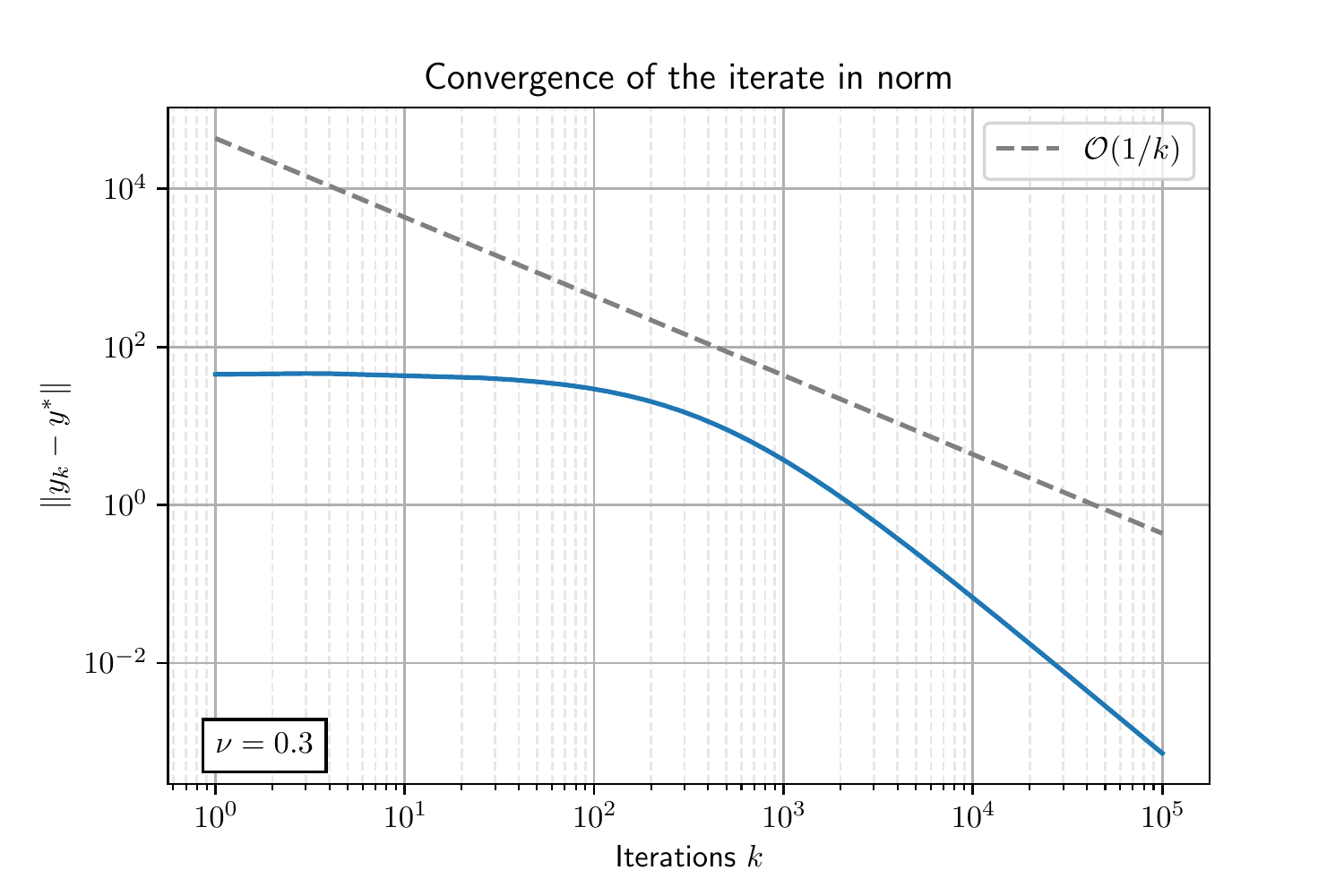}
	\caption{Convergence of the minimax gap like $ \mathcal{O} (\frac{1}{K^{2}}) $ and of the sequence $(y_{k})_{k \geq 0} $ in norm like $ \mathcal{O} (\frac{1}{K}) $ for $ \nu > 0 $.}
	\label{fig:toy_nu>0}
\end{figure}

Let $ (x^{\ast}, y^{\ast}) \in \R^{d} \times \R^{n} $ be a saddle point. In the case $ \nu > 0 $, i.e., the regulariser $ g $ being $ \nu $-strongly convex, we proved strong non-asymptotic convergence of the sequence $ (y_{k})_{k \geq 0} \to y^{\ast} $ like $ \mathcal{O} (\frac{1}{K}) $ and convergence of the minimax gap at the ergodic sequences to zero like $ \mathcal{O} (\frac{1}{K^{2}}) $. The numerical behaviour of our method validating the theoretical claims for $ \nu > 0 $ is highlighted in Figure~\ref{fig:toy_nu>0}. The plots shown are for a single random initialisation and with the choice $ \nu = \frac{3}{10} $.

\subsection{Multi kernel support vector machine}
The second application to test our method in practice is to learn a combined kernel matrix for a multi kernel \emph{support vector machine} (SVM). We have a set of labelled training data
\begin{equation}
	S_{n}
	= \{ (a_{1}, b_{1}), \ldots, (a_{n}, b_{n}) \}
	\subseteq \R^{m} \times \{ -1, 1 \},
\end{equation}
where we call $ b = (b_{i})_{i = 1}^{n} $, and a set of unlabelled test data
\begin{equation}
	T_{l}
	= \{ a_{n + 1}, \ldots, a_{n + l} \}
	\subseteq \R^{m}.
\end{equation}
We consider embeddings of the data according to a kernel function $ \kappa : \R^{m} \times \R^{m} \to \R $ with the corresponding symmetric and positive semidefinite kernel matrix
\begin{equation}
	\mathcal{K} =
	\left(
		\begin{array}{cc}
			\mathcal{K}^{tr} & \mathcal{K}^{tr, t}\\
			\mathcal{K}^{t,tr} & \mathcal{K}^{t}
		\end{array}
	\right),
\end{equation}
where $ \mathcal{K}_{ij} = \kappa(a_{i}, a_{j}) $ for $ i, \, j = 1, \ldots, n, n + 1, \ldots, n+l $.

In the following $ e $ is a vector  of appropriate size consisting of ones. According to~\cite{MKSVM} the problem of interest is
\begin{equation}
	\label{mksvm-original}
	\min_{\substack{\mathcal{K} \in \mathbb{K} \\ \tr (\mathcal{K}) = c }}
	\hspace{2mm}
	\max_{\substack{0 \leqq \alpha \leqq C \\ \sprod{\alpha}{b} = 0}}
	\,
	\alpha^{T}e
	- \frac{1}{2} \alpha^{T} G(\mathcal{K}^{tr}) \alpha
	- \frac{\nu}{2} \normsq{\alpha}_{2},
\end{equation}
where $ \mathbb{K} $ is the model class of kernel matrices, $ c \in (0, +\infty) $, $ C \in (0, +\infty] $ and $ \nu \in [0, +\infty) $ are model parameters and we define $ G(\mathcal{K}^{tr}) := \diag(b) \mathcal{K}^{tr} \diag(b) $.

The set $ \mathbb{K} $ is restricted to be the set of positive semidefinite matrices that can be written as a non negative linear combination of kernel matrices $ \mathcal{K}_{1}, \ldots, \mathcal{K}_{d} $, i.e.,
\begin{equation}
	\mathbb{K} =
	\left\{
		\mathcal{K} \in S^m_+ \ \bigg | \
		\mathcal{K} = \sum_{i=1}^{d} \eta_{i} \mathcal{K}_{i}, \,
		\eta_{i} \geq 0 \text{ for } i=1, ..., d
	\right\}.
\end{equation}
With this choice~\eqref{mksvm-original} becomes
\begin{equation}
	\label{mksvm-2}
	\min_{\substack{\sprod{\eta}{r} = c \\ \eta \geqq 0 }}
	\hspace{2mm}
	\max_{\substack{0 \leqq \alpha \leqq C \\ \sprod{\alpha}{b} = 0}}
	\,
	\alpha^{T}e
	- \frac{1}{2} \sum_{i=1}^{d} \eta_{i} \alpha^{T} G(\mathcal{K}^{tr}_{i}) \alpha
	- \frac{\nu}{2} \normsq{\alpha},
\end{equation}
where $ \eta = (\eta_{i})_{i=1}^{d} $ and $ r = (r_{i})_{i=1}^{d} $ with $ r_{i} = \tr (\mathcal{K}_{i}) $ for $i=1, ..., d$. Assume $ (\eta^{\ast}, \alpha^{\ast}) \in \R^{d} \times \R^{n} $ to be a saddle point of~\eqref{mksvm-2} and write
\begin{equation}
	\mathcal{K}^{\ast} = \sum_{j=1}^{d} \eta^{\ast}_{j} \mathcal{K}_{j}.
\end{equation}
Following the considerations of~\cite{Aybat} we compute for $ a_{k} \in T_{l} $ with $ k \in \{ n+1, \ldots, n+l \} $,
\begin{equation}
	\label{mksvm-predict}
	\mathcal{L} (a_{k})
		= \sgn \left(
			\sum_{i = 1}^{n} b_{i} \alpha^{\ast}_{i} \mathcal{K}^{\ast}_{ik} + \gamma
		\right)
		= \sgn \left(
			\sum_{i = 1}^{n} \sum_{j=1}^{d} b_{i} \alpha^{\ast}_{i} \eta^{\ast}_{j} \left( \mathcal{K}_{j} \right)_{ik} + \gamma
		\right),
\end{equation}
with
\begin{equation}
	\gamma
		= b_{j_{0}} (1 - \nu \alpha^{\ast}_{j_{0}})
			- \sum_{i=1}^{n} b_{i} \alpha^{\ast}_{i} \mathcal{K}^{\ast}_{i j_{0}}
		= b_{j_{0}} (1 - \nu \alpha^{\ast}_{j_{0}})
			- \sum_{i=1}^{n} \sum_{j=1}^{d} b_{i} \alpha^{\ast}_{i} \eta^{\ast}_{j} \left( \mathcal{K}_{j} \right)_{ij_{0}},
\end{equation}
for some $ j_{0} \in \{ 1, \ldots, n \} $ such that $ 0 < \alpha^{\ast}_{j_{0}} < C $.

After writing $ x_{i} = \frac{r_{i} \eta_{i}}{c} $ for $i=1, ..., d$ and augmenting the objective with an additional (strongly) convex penalisation term, we obtain
\begin{equation}
	\label{mksvm}
	\min_{x \in \R^{d}} \max_{y \in \R^{n}} \,
	\delta_{\Delta}(x) + \frac{\mu}{2} \normsq{x}
	- \frac{1}{2} \sum_{i=1}^{d} x_{i} y^{T} M_{i} y
	+ y^{T} e
	- \left (\delta_{Y}(y) + \frac{\nu}{2} \normsq{y} \right),
\end{equation}
where $ \mu \geq 0 $ and $ M_{i} := \frac{c}{r_{i}} G(\mathcal{K}_{i}^{tr}) $ for $i=1, ..., d$,
\begin{equation}
	\Delta := \{ x \in \R^{d} \ | \ x \geqq 0, \, \sprod{x}{e} = 1 \}
\end{equation}
is the $ m $-dimensional unit simplex and
\begin{equation}
	Y := \{ y \in \R^{n} \ | \ 0 \leqq y \leqq C, \, \sprod{y}{b} = 0 \}
\end{equation}
is the intersection of a box and a hyperplane.

In the notation of~\eqref{min-max} we have $ \Phi : \R^{d} \times \R^{n} \to \R \cup \{ + \infty \} $ defined by
\begin{equation}
	\Phi (x, y) =
		\delta_{\Delta}(x)
		+ \frac{\mu}{2} \normsq{x}
		- \frac{1}{2} \sum_{i=1}^{d} x_{i} y^{T} M_{i} y
		+ y^{T} e,
\end{equation}
and $ g: \R^{n} \to \R \cup \{ +\infty \} $ given by
\begin{equation}
	g(y) = \delta_{Y}(y) + \frac{\nu}{2} \normsq{y}.
\end{equation}
We see that $ \Phi $ and $g$ satisfy the assumptions considered for problem \eqref{min-max}.

The algorithm~\eqref{alg_y}-\eqref{alg_x} iterates as follows for $ k \geq 0 $
\begin{equation}
	\begin{split}
		\left\{
		\begin{array}{rcl}
			v_{k} &=& y_{k} + \sigma_{k} \left[ (1 + \theta_{k}) \ny \Phi(x_{k}, y_{k}) - \theta_{k} \ny \Phi(x_{k-1}, y_{k-1})\right],\vspace{1mm}\\
			y_{k+1}  &=& \prox{\sigma_{k} g}{v_{k}} = P_{Y} \left( \frac{1}{1 + \nu \sigma_{k}} v_{k} \right),\vspace{1mm}\\
			x_{k+1}  &=& \prox{\tau_{k} \Phi(\ph, y_{k+1})}{x_{k}} = P_{\Delta} \left( \frac{1}{1 + \mu \tau_{k}} (x_{k} + \tau_{k} \xi^{y_{k+1}}) \right),
		\end{array}
		\right.
	\end{split}
\end{equation}
where
\begin{equation}
	\ny \Phi(x, y)
		= \minus \left( \frac{1}{2} \sum_{i = 1}^{d} x_{i} (M_{i} + M_{i}^{T}) \right)y + e
		= \minus \left( \sum_{i = 1}^{d} x_{i} M_{i} \right)y + e \ \mbox{ for }(x,y) \in \Delta \times \R^n
\end{equation}
and
\begin{equation}
	\xi^{y} := \left( \frac{1}{2} y^{T} M_{i} y \right)_{i=1}^{d}.
\end{equation}

To determine the correct step sizes and momentum parameter, we need to find Lipschitz constants for $ \ny \Phi $, i.e., $ L_{yx} $, $ L_{yy} \geq 0 $ such that~\eqref{cond-Phi_y} holds. Recall, that we require for all $ (x,y), \, (x',y') \in \pr_{\cH} (\dom \Phi) \times \dom g $
\begin{equation}
	\norm{\ny \Phi (x, y) - \ny \Phi (x', y')} \leq L_{yx} \norm{x-x'} + L_{yy} \norm{y - y'},
\end{equation}
with $ \pr_{\cH} (\dom \Phi) = \Delta $ and $ \dom g = Y $.

Let $ (x,y), \, (x',y') \in \Delta \times Y $. Then
\begin{equation}
	\begin{split}
		\norm{\ny \Phi(x, y) - \ny \Phi(x', y')} &= \norm{
				- \sum_{i = 1}^{d} x_{i} M_{i} y + e
				+ \sum_{i = 1}^{d} x'_{i} M_{i} y' - e}\\
			&= \norm{
				\sum_{i = 1}^{d} x_{i} M_{i} y'
				- \sum_{i = 1}^{d} x_{i} M_{i} y
				+ \sum_{i = 1}^{d} x'_{i} M_{i} y'
				- \sum_{i = 1}^{d} x_{i} M_{i} y'}\\
			&\leq \norm{\sum_{i = 1}^{d} x_{i} M_{i} (y - y')}
				+ \norm{\sum_{i = 1}^{d} (x_{i} - x'_{i}) M_{i} y'}\\
			&\leq \sum_{i = 1}^{d} \lvert x_{i} \rvert \norm{M_{i}} \norm{y - y'}
				+ \sum_{i = 1}^{d} \lvert x_{i} - x'_{i} \rvert \norm{M_{i}} \norm{y'}\\
			&\leq \left( \norm{x}_{1} \max_{1 \leq i \leq d} \norm{M_{i}} \right) \norm{y - y'}
				+  \left( \norm{y'} \max_{1 \leq i \leq d} \norm{M_{i}} \right) \norm{x - x'}_{1}\\
			&\leq \left( \norm{x}_{1} \max_{1 \leq i \leq d} \norm{M_{i}} \right) \norm{y - y'}
				+  \left( \norm{y'} \sqrt{d} \max_{1 \leq i \leq d} \norm{M_{i}} \right) \norm{x - x'}.
	\end{split}
\end{equation}
As $ x \in \Delta $, we have $ \norm{x}_{1} = 1 $ and since $ y' \in Y $ we get $ \norm{y'} \leq C \sqrt{n} $. Thus we obtain
\begin{equation}
	\norm{\ny \Phi(x, y) - \ny \Phi(x', y')} \leq L_{yx} \norm{x-x'} + L_{yy} \norm{y - y'},
\end{equation}
with
\begin{equation}
	L_{yx} = C \sqrt{d n} \max_{1 \leq i \leq d} \norm{M_{i}},
	\hspace{5mm}
	L_{yy} = \max_{1 \leq i \leq d} \norm{M_{i}}.
\end{equation}

For our experiments we use four different data sets from the ``UCI Machine Learning Repository''~\cite{uci}: the (original) Wisconsin \emph{breast cancer} dataset~\cite{breast} (699 total observations including 16 incomplete examples; 9 features), the Statlog \emph{heart disease} data set (270 observations; 13 features), the \emph{Ionosphere} data set (351 observations; 33 features) and the Connectionist Bench \emph{Sonar} data set (208 observations; 60 features). All the data sets are normalised such that each feature column has zero mean and standard deviation equal to one.

Furthermore we take $ d = 3 $ given kernel functions, namely a polynomial kernel function $ k_{1}(a, a') = (1 + a^{T}a')^{2} $ of degree 2 for $ \mathcal{K}_{1} $, a Gaussian kernel function $ k_{2}(a, a') = \exp (\minus \frac{1}{2}(a - a')^{T}(a - a')/\frac{1}{10}) $ for $ \mathcal{K}_{2} $ and a linear kernel function $ k_{3} (a, a') = a^{T}a' $ for $ \mathcal{K}_{3} $. The resulting kernel matrices are normalised according to~\cite[Section~4.8]{MKSVM}, giving
\begin{equation}
	r_{i} = \tr (\mathcal{K}_{i}) = n + l.
\end{equation}
The model parameter $ c > 0 $ is chosen to be
\begin{equation}
	c = \sum_{i = 1}^{d} r_{i} = d(n + l),
\end{equation}
and we set $ C = 1 $.

On this application we test the three proposed versions of OGAProx. We refer to the version of OGAProx with constant parameters from Section~\ref{convex-concave-setting} as OGAProx-C1, to the one with adaptive parameters from Section~\ref{convex-strconcave} as OGAProx-A and to the one from Section~\ref{strconvex-strconcave} giving linear convergence with constant parameters as OGAProx-C2. The results are compared with those obtained by APD1 and APD2 from~\cite{Aybat}. In their experiments on multi kernel SVMs they showed superiority of their method compared to \emph{Mirror Prox} by~\cite{MirrorProx} in terms of accuracy, runtime and relative error. They also argued that with APD they are able to obtain decent approximations of solutions of~\eqref{mksvm-2} by interior point methods such as MOSEK~\cite{mosek} taking about the same amount of runtime.

The main difference between APD and our method OGAProx is that for the first a gradient step in the first component is employed whereas for the latter a purely proximal step is used. To be able to employ APD2 with adaptive parameters for $ \nu > 0 $, the roles of $ x $ and $ y $ in~\eqref{mksvm} have to be switched, giving a different method than OGAProx-A. The runtime of both methods however is still very similar as both use the same number of gradient computations/storages and projections per iteration.

All algorithms are initialised with
\begin{equation}
	x_{0} = x_{-1} = \frac{1}{d} e,
		\hspace{5mm}
	y_{0} = y_{-1} = 0.
\end{equation}

Each data set is randomly partitioned into 80~\% training and 20~\% test set. The test set is used to judge the quality of the obtained model by predicting the labels via~\eqref{mksvm-predict} and computing the resulting test set accuracy (TSA).
Note that the TSA is not guaranteed to converge or increase at all by our theoretical considerations, which only state convergence of the iterates and in terms of function values.
The reported TSA values are the average over 10 random partitions. Due to occasionally occurring rather dramatic deflections of the TSA we actually compute 12 runs, but remove minimum and maximum values before calculating the mean.

\subsubsection{1-norm soft margin classifier}
For $ \mu = \nu = 0 $ the formulation~\eqref{mksvm-2} realises the so-called 1-norm soft margin classifier. In this case $ g $ is merely convex and we can only use the constant parameter choice from Section~\ref{convex-concave-setting} with the name OGAProx-C1. We compare the results with those obtained by APD1 from~\cite{Aybat}.

\begin{table}[h!]
	\centering
	\begin{tabular}{@{\extracolsep{4pt}}llccccc@{}}
	    \toprule
	    & & \multicolumn{5}{c}{TSA at iteration $ k $} \\
	    Method & Data set & $ k = 250 $ & $ k = 500 $ & $ k = 1000 $ & $ k = 1500 $ & $ k = 2000 $ \\
	    \midrule
	    \multirow{4}{*}{OGAProx-C1}
	    & Breast cancer & 97.15 & 97.37 & 97.08 & 93.94 & \textbf{97.45} \\
	    & Heart disease & 74.63 & 74.07 & 80.00 & 81.30 & \textbf{82.78} \\
	    & Ionosphere    & 70.85 & 85.35 & 90.28 & 87.46 & \textbf{93.24} \\
	    & Sonar         & 70.00 & 75.24 & 83.81 & 84.52 & \textbf{85.95}\vspace{3mm} \\
		\multirow{4}{*}{APD1}
		& Breast cancer & 97.23 & 97.37 & \textbf{97.45} & 94.01 & \textbf{97.45} \\
		& Heart disease & 74.63 & 72.59 & 81.85 & 80.74 & 82.41 \\
		& Ionosphere    & 70.85 & 85.35 & 85.49 & 88.73 & 92.68 \\
		& Sonar         & 70.00 & 74.76 & 81.67 & 84.76 & 84.52 \\
		\bottomrule
	\end{tabular}
	\vspace{0.2cm}
	\caption{
		TSA of 1-norm soft margin classifier ($ \mu = 0 $, $ \nu = 0 $, $ C = 1 $) trained with OGAProx-C1 and APD1, averaged over 10 random partitions.}
	\label{tab:convex-concave}
\end{table}

In the case of 1-norm soft margin classifier the results reported in Table~\ref{tab:convex-concave} paint a clear picture. OGAProx outperforms APD on three out of four data sets and ties on one data set, achieving maximum TSA values of 97.45~\%, 82.78~\%, 93.24~\% and 85.95~\% on Breast cancer, Heart disease, Ionosphere and Sonar, respectively.

\subsubsection{2-norm soft margin classifier}
For $ \mu = 0 $ and $ \nu > 0 $ from~\eqref{mksvm-2} we obtain the so-called 2-norm soft margin classifier with $ C = 1 $. In this case $ g $ is $ \nu $-strongly convex and we can use both parameter choices from Section~\ref{convex-concave-setting} and the one from Section~\ref{convex-strconcave} giving OGAProx-C1 and OGAProx-A, respectively. This time we compare the results with those obtained by APD1 as well as APD2 from~\cite{Aybat}.

\begin{table}[h!]
	\centering
	\begin{tabular}{@{\extracolsep{4pt}}llccccc@{}}
	    \toprule
	    & & \multicolumn{5}{c}{TSA at iteration $ k $} \\
	    Method & Data set & $ k = 250 $ & $ k = 500 $ & $ k = 1000 $ & $ k = 1500 $ & $ k = 2000 $ \\
	    \midrule
	    \multirow{4}{*}{OGAProx-C1}
	    & Breast cancer & 97.15 & 97.37 & 97.15 & \textbf{97.45} & 97.15 \\
	    & Heart disease & 75.19 & 75.00 & 77.78 & 83.52 & 83.52 \\
	    & Ionosphere    & 70.99 & 85.35 & 89.86 & 87.89 & 91.27 \\
	    & Sonar         & 70.71 & 77.86 & 81.90 & 85.71 & \textbf{86.19}\vspace{3mm} \\
		\multirow{4}{*}{APD1}
		& Breast cancer & 97.23 & 97.37 & 97.30 & 97.37 & 97.37 \\
		& Heart disease & 75.37 & 67.78 & 80.74 & 82.22 & \textbf{84.81} \\
		& Ionosphere    & 71.27 & 85.35 & 88.87 & 89.72 & \textbf{92.39} \\
		& Sonar         & 70.48 & 76.43 & 83.33 & 84.76 & 85.71 \\
		\midrule
		\multirow{4}{*}{OGAProx-A}
		& Breast cancer & 97.15 & 97.37 & 97.37 & 97.45 & 97.45 \\
		& Heart disease & 76.11 & 73.70 & 83.70 & 81.30 & \textbf{84.26} \\
		& Ionosphere    & 70.85 & 85.21 & 86.34 & 90.42 & \textbf{93.52} \\
		& Sonar         & 70.48 & 76.90 & 83.33 & 82.62 & 84.76\vspace{3mm} \\
		\multirow{4}{*}{APD2}
		& Breast cancer & 97.23 & 97.37 & \textbf{97.59} & 97.01 & 96.72 \\
		& Heart disease & 76.11 & 71.30 & 81.48 & 78.70 & 83.15 \\
		& Ionosphere    & 71.13 & 85.35 & 84.79 & 84.93 & 90.42 \\
		& Sonar         & 70.24 & 75.95 & 84.05 & 84.52 & \textbf{86.19} \\
		\bottomrule
	\end{tabular}
	\vspace{0.2cm}
	\caption{
		TSA of 2-norm soft margin classifier ($ \mu = 0 $, $ \nu = \frac{1}{2} $, $ C = 1 $) trained with OGAProx-C1, OGAProx-A, APD1 and APD2, averaged over 10 random partitions.}
	\label{tab:convex-strconcave}
\end{table}

We see in Table~\ref{tab:convex-strconcave} that the situation for the 2-norm soft margin classifier is more diverse than previously with the 1-norm soft margin classifier. Comparing the two constant methods -- OGAProx-C1 and APD1 -- with each other, as well as the two adaptive methods -- OGAProx-A and APD2 -- we see that in both cases two out of four times OGAProx is better than APD and vice versa.
Notice that the two data sets with in general lower TSA, namely Heart disease and Sonar, seem to benefit from the regularising effect of $ \nu > 0 $, while those with already very good results on the other hand do not, compared to the results of the 1-norm soft margin classifier with $ \nu = 0 $.
In addition note that the adaptive variant OGAProx-A improves on the result of OGAProx-C1 on three out of four data sets.

\subsubsection{Regularised 2-norm soft margin classifier}
For $ \mu > 0 $ and $ \nu > 0 $ from~\eqref{mksvm-2} we again obtain the so-called 2-norm soft margin classifier with $ C = 1 $, this time, however, in a regularised version. Now not only  $ g $ is strongly convex, but also $ \Phi(\ph, y) $ and we can use all our parameter choices from Section~\ref{convex-concave-setting}, Section~\ref{convex-strconcave} and Section~\ref{strconvex-strconcave} yielding OGAProx-C1, OGAProx-A and OGAProx-C2, respectively. Once more we compare the results with those obtained by APD1 as well as APD2 from~\cite{Aybat}, pointing out that that OGAProx-C2 has no APD counterpart harnessing the additional strong convexity of the problem.

\begin{table}[h!]
	\centering
	\begin{tabular}{@{\extracolsep{4pt}}llccccc@{}}
	    \toprule
	    & & \multicolumn{5}{c}{TSA at iteration $ k $} \\
	    Method & Data set & $ k = 250 $ & $ k = 500 $ & $ k = 1000 $ & $ k = 1500 $ & $ k = 2000 $ \\
	    \midrule
	    \multirow{4}{*}{OGAProx-C1}
	    & Breast cancer & 97.15 & 97.37 & 97.15 & \textbf{97.52} & 97.45 \\
	    & Heart disease & 75.19 & 73.52 & 77.22 & 83.15 & 83.70 \\
	    & Ionosphere    & 70.99 & 85.35 & 87.89 & 91.41 & \textbf{91.97} \\
	    & Sonar         & 70.48 & 78.81 & 83.33 & 84.76 & \textbf{85.95}\vspace{3mm} \\
		\multirow{4}{*}{APD1}
		& Breast cancer & 97.23 & 97.37 & 97.37 & 97.01 & 97.30 \\
		& Heart disease & 75.19 & 68.89 & 75.56 & 79.81 & \textbf{84.07} \\
		& Ionosphere    & 71.27 & 85.35 & 86.06 & 89.15 & 91.69 \\
		& Sonar         & 70.71 & 76.43 & 83.10 & 85.48 & 85.48 \\
		\midrule
		\multirow{4}{*}{OGAProx-A}
		& Breast cancer & 97.15 & 97.37 & 97.45 & 97.37 & 97.30 \\
		& Heart disease & 76.11 & 70.93 & 82.78 & 80.74 & 83.52 \\
		& Ionosphere    & 70.85 & 85.21 & 85.92 & 89.86 & \textbf{93.38} \\
		& Sonar         & 70.24 & 76.43 & 82.86 & 86.19 & 86.19\vspace{3mm} \\
		\multirow{4}{*}{APD2}
		& Breast cancer & 97.23 & 97.37 & 97.45 & 94.53 & \textbf{97.52} \\
		& Heart disease & 76.11 & 71.67 & 80.00 & 79.26 & \textbf{83.52} \\
		& Ionosphere    & 71.13 & 85.35 & 86.90 & 92.39 & 91.13 \\
		& Sonar         & 70.24 & 75.00 & 82.62 & 84.52 & \textbf{86.43} \\
		\midrule
		\multirow{4}{*}{OGAProx-C2}
		& Breast cancer & 97.15 & 97.45 & \textbf{97.59} & 97.15 & 96.57 \\
		& Heart disease & 74.07 & 78.52 & 76.11 & 82.22 & \textbf{83.70} \\
		& Ionosphere    & 70.42 & 84.37 & 86.48 & 90.85 & \textbf{92.25} \\
		& Sonar         & 69.05 & 74.29 & 85.24 & 85.71 & \textbf{86.19} \\
		\bottomrule
	\end{tabular}
	\vspace{0.2cm}
	\caption{
		TSA of regularised 2-norm soft margin classifier ($ \mu = 1 $, $ \nu = \frac{1}{2} $, $ C = 1 $) trained with OGAProx-C1, OGAProx-A, OGAProx-C2, APD1 and APD2, averaged over 10 random partitions.}
	\label{tab:strconvex-strconcave}
\end{table}

We see in Table~\ref{tab:strconvex-strconcave} that for the regularised 2-norm soft margin classifier the situation is similar to the version without additional regulariser. This time for the constant methods, OGAProx-C1 and APD1, OGAProx is better than APD on three data sets while APD is better than OGAProx on only one. On the contrary, for the adaptive methods, OGAProx-A and APD2, it is the other way round. APD performs better than APD on three data sets while OGAProx is better than APD on only one. For the second version of OGAProx with constant parameter choice exhibiting linear convergence in both iterates and function values, there is no APD counterpart. When we compare the results for OGAProx-C2 to those of OGAProx-C1, then we see that the TSA values become better in general with improvements on three out of four data sets and one draw. On the Breast cancer data set OGAProx-C2 even delivers the maximum TSA over all considered methods.

\subsection{Classification incorporating minimax group fairness}
We want to classify labelled data $ (a_{j}, b_{j})_{j=1}^{n} \subseteq \R^{d} \times \{ \pm 1 \} $, additionally taking into account so-called \emph{minimax group fairness}~\cite{ParetoFair, GroupFair}. The data is divided into $ m $ groups $ G_{1}, ..., G_{m} $, such that for $ i \in [m] := \{ 1, ..., m \} $ we have $ G_{i} = (a_{i_{j}}, b_{i_{j}})_{j=1}^{n_{i}} \subseteq (a_{j}, b_{j})_{j=1}^{n} $ with $ n_{i} := \abs{G_{i}} $ and $ i_{j} \in [n] $ for all $ i \in [m] $ and all $ j \in [n_{i}] $.
\emph{Fairness} is measured by worst-case outcomes across the considered groups. Hence we consider the following problem,
\begin{equation}\label{fair-minimax}
	\min_{x \in \R^d} \max_{i \in [m]} f_{i}(x),
\end{equation}
with
\begin{equation}
	f_{i}(x) = \frac{1}{n_{i}} \sum_{j=1}^{n_{i}} L(h_{x}(a_{i_{j}}), b_{i_{j}}),
\end{equation}
where $ h_{x} $ is a function parametrised by $ x $, mapping features to predicted labels, and $ L $ is a loss function measuring the error between the predicted and true labels.

It is easy to see that~\eqref{fair-minimax} is equivalent to
\begin{equation}
	\min_{x \in \R^d} \max_{y \in \Delta_{m}} \sum_{i=1}^{m} y_{i} f_{i}(x),
\end{equation}
where $ \Delta_{m} := \{ (v_{1}, ..., v_m) \in \R^{m} \ | \ \sum_{i=1}^{m} v_{i} = 1, \; v_{i} \geq 0 \text{ for } i = 1, ..., m \} $ denotes the probability simplex in $ \R^{m} $.
We will work with a linear (affine) predictor $ h_{x}: \R^{d} \to \R $ given by
\begin{equation}
	h_{x}(a) = a^{T}x,
\end{equation}
with $ x \in \R^{d} $ and $ L: \R \times \R \to \R $ being the hinge loss, i.e.,
\begin{equation}
	L(r,s) = \max \{ 0, 1 - sr \},
\end{equation}
for $ r $, $ s \in \R $.

Combining all of the above we get
\begin{equation}\label{fair-minimax-nice}
	\min_{x \in \R^{d}} \max_{y \in \R^{m}} \Phi (x, y) - g(y),
\end{equation}
with $ \Phi: \R^{d} \times \R^{m} \to \R $ defined by
\begin{equation}
	\Phi (x, y) = \sum_{i=1}^{m} y_{i} \frac{1}{n_{i}} \sum_{j=1}^{n_{i}} \max \{ 0, 1 - b_{i_{j}} a_{i_{j}}^{T} x \},
\end{equation}
and $ g: \R^{m} \to \R \cup \{ + \infty \} $ given by
\begin{equation}
	g(y) = \delta_{\Delta_{m}}(y).
\end{equation}
The function $ g $ is proper,  lower semicontinuous and convex (with modulus $ \nu = 0 $). Furthermore we observe that $ \Phi (\cdot, y): \R^{d} \to \R $ is proper, convex and lower semicontinuous for all $ y \in \dom g = \Delta_{m} $ and for all $ x \in \Pr_{\R^{d}} (\dom \Phi) = \R^{d} $ we have $ \dom \Phi (x, \cdot) = \R^{m} $ and $ \Phi (x, \cdot): \R^{m} \to \R $ is concave and Fr{\'e}chet differentiable. However, note  that $ \Phi $ is not differentiable in its first component.

Moreover the Lipschitz condition on the gradient is fulfilled as well. Indeed, for $ (x,y), \, (x',y') \in \R^{d} \times \Delta_{m} $ we have
\begin{equation}
	\norm{ \nabla_{y} \Phi (x,y) - \nabla_{y} \Phi (x',y') } \leq L_{yx} \norm{x-x'} + L_{yy} \norm{y-y'},
\end{equation}
with
\begin{equation}
	L_{yx} = \sqrt{ \sum_{i=1}^{m} \frac{1}{n_{i}} \sum_{j=1}^{n_{i}} \normsq{a_{i_{j}}} } \quad \mbox{and} \quad
	L_{yy} = 0.
\end{equation}

Additionally, with $ \tau > 0 $ and $ y \in \dom g $, we have for $ x \in \R^{d} $
\begin{equation}\label{fair-prox}
	\prox{\tau \Phi(\cdot, y)}{x} =
		\argmin_{u \in \R^{d}}
			\left\{
				\tau \sum_{i=1}^{m} y_{i} \frac{1}{n_{i}} \sum_{j=1}^{n_{i}} \max \{ 0, 1 - b_{i_{j}} a_{i_{j}}^{T} u \}
				+ \frac{1}{2} \normsq{u - x}
			\right\}.\\
\end{equation}
By introducing slack variables for the pointwise maximum, we see that the above minimisation problem is equivalent to the following quadratic program
\begin{equation}
\begin{split}
	\begin{aligned}[t]
		\min_{\substack{u \in \R^{d},\\ r_{ij} \in \R, \\ i \in [m], \; j \in [n_{i}]}} &  \quad
		\left\{
			\tau \sum_{i=1}^{m} \sum_{j=1}^{n_{i}} y_{i} \frac{1}{n_{i}}
				r_{ij}
			+ \frac{1}{2} \normsq{u - x}
		\right\}.\\
		\textrm{s.t.} \quad & r_{ij} \geq 0 \quad \forall \, i \in [m], \, \forall \, j \in [n_{i}]\\
		& r_{ij} + b_{i_{j}} a_{i_{j}}^{T} u \geq 1 \quad \forall \, i \in [m], \, \forall \, j \in [n_{i}]\\
	\end{aligned}
\end{split}
\end{equation}

For our practical applications we consider the Statlog \emph{heart disease} data set (270 observations; 13 features) from the ``UCI Machine Learning Repository''~\cite{uci} and consider two different groupings; one consisting of the sex of the patients, while the other one is regarding the patients' age. For ``sex'' we have two groups, that is female patients (Group~S1) and male patients (Group~S2), whereas for ``age'' we consider three groups, that is patients that are younger than 50 years old (Group~A1), patients that are younger than 60 but at least 50 years old (Group~A2), and patients that are 60 years of age or older (Group~A3). The data set is randomly partitioned into 80~\% training data and 20~\% test data. The results in Table~\ref{tab:sex} and Table~\ref{tab:age} are the values of the achieved test set accuracy (TSA) averaged over 5 random partitions. For each considered group we state the intragroup TSA together with the overall TSA for the entire test set.

Every time we report the results obtained by iterates of OGAProx governed by solving the minimax problem~\eqref{fair-minimax-nice} taking into account the considered groups (``with fairness''), as well as the results obtained by not taking into account minimax group fairness (``without fairness''), i.e., solving the problem for a single extensive group $ G_{1} = (a_{j}, b_{j})_{j=1}^{n} $ with $ n_{1} = n $, yielding the minimisation of the average loss over the whole population and leading to an ``ordinary'' minimisation problem.

We see in Table~\ref{tab:sex} and Table~\ref{tab:age} that taking into account the groups regarding ``sex'' and ``age'', respectively, is beneficial for training the affine classifier. In both cases ``with fairness'' achieves the highest TSA for each group and at the same time the highest overall TSA as well.\vspace{1ex}

\begin{table*}\centering
	\begin{tabular}{@{\extracolsep{4pt}}rccccccccc@{}}
	\toprule
	& \phantom{} & \multicolumn{2}{c}{Group~S1} & \phantom{} & \multicolumn{2}{c}{Group~S2} & \phantom{ } & \multicolumn{2}{c}{Overall}\\
	\cmidrule{3-4} \cmidrule{6-7} \cmidrule{9-10}
	&& \multirow{2}{*}{\shortstack[c]{with\\fairness}} & \multirow{2}{*}{\shortstack[c]{without\\fairness}} && \multirow{2}{*}{\shortstack[c]{with\\fairness}} & \multirow{2}{*}{\shortstack[c]{without\\fairness}} && \multirow{2}{*}{\shortstack[c]{with\\fairness}} & \multirow{2}{*}{\shortstack[c]{without\\fairness}} \\
	$ k $&&&&&&&&&\\
	\midrule
	100		&& \textbf{95.78} & \textbf{95.78} && 80.68				& 80.84 && 85.56			& 85.56 \\
%	200		&& \textbf{95.78} & \textbf{95.78} && 80.68				& 80.84 && 85.56			& 85.56 \\
%	300		&& \textbf{95.78} & \textbf{95.78} && 80.12				& 80.84 && 85.19			& 85.56 \\
%	400		&& \textbf{95.78} & \textbf{95.78} && 80.68				& 80.28 && 85.56			& 85.19 \\
	500		&& \textbf{95.78} & \textbf{95.78} && \textbf{81.15}	& 80.28 && \textbf{85.93}	& 85.19 \\
%	600		&& \textbf{95.78} & \textbf{95.78} && \textbf{81.15}	& 80.28 && \textbf{85.93}	& 85.19 \\
%	700		&& \textbf{95.78} & \textbf{95.78} && \textbf{81.15}	& 80.28 && \textbf{85.93}	& 85.19 \\
%	800		&& \textbf{95.78} & \textbf{95.78} && \textbf{81.15}	& 80.28 && \textbf{85.93}	& 85.19 \\
%	900		&& \textbf{95.78} & \textbf{95.78} && \textbf{81.15}	& 80.28 && \textbf{85.93}	& 85.19 \\
	1000	&& \textbf{95.78} & \textbf{95.78} && \textbf{81.15}	& 80.28 && \textbf{85.93}	& 85.19 \\
	\bottomrule
	\end{tabular}
	\caption{TSA of the affine classifier after $ k $ iterations of OGAProx for the groups according to ``sex'', averaged over 5 random partitions.}
	\label{tab:sex}
\end{table*}

\begin{table*}\centering
	\begin{tabular}{@{\extracolsep{2pt}}rcccccccccccc@{}}
	\toprule
	& \phantom{} & \multicolumn{2}{c}{Group~A1} & \phantom{}& \multicolumn{2}{c}{Group~A2} &
	\phantom{} & \multicolumn{2}{c}{Group~A3} & \phantom{ } & \multicolumn{2}{c}{Overall}\\
	\cmidrule{3-4} \cmidrule{6-7} \cmidrule{9-10} \cmidrule{12-13}
	&& \multirow{2}{*}{\shortstack[c]{with\\fairness}} & \multirow{2}{*}{\shortstack[c]{without\\fairness}} && \multirow{2}{*}{\shortstack[c]{with\\fairness}} & \multirow{2}{*}{\shortstack[c]{without\\fairness}} && \multirow{2}{*}{\shortstack[c]{with\\fairness}} & \multirow{2}{*}{\shortstack[c]{without\\fairness}} && \multirow{2}{*}{\shortstack[c]{with\\fairness}} & \multirow{2}{*}{\shortstack[c]{without\\fairness}} \\
	$ k $&&&&&&&&&&&&\\
	\midrule
	100		&& 87.76 			& 86.48 && 82.97			& 82.97 && \textbf{86.93}	& \textbf{86.93} && 85.93			& 85.56 \\
%	200		&& 87.76 			& 86.48 && 83.35 			& 82.97 && 85.75			& \textbf{86.93} && 85.56			& 85.56 \\
%	300		&& 87.76 			& 86.48 && \textbf{83.84}	& 82.97 && \textbf{86.93}	& \textbf{86.93} && 86.30			& 85.56 \\
%	400 	&& \textbf{88.71} 	& 85.53 && \textbf{83.84}	& 82.97 && \textbf{86.93}	& \textbf{86.93} && \textbf{86.67}	& 85.19 \\
	500		&& \textbf{88.71}	& 85.53 && \textbf{83.84}	& 82.97 && \textbf{86.93}	& \textbf{86.93} && \textbf{86.67}	& 85.19 \\
%	600		&& \textbf{88.71}	& 85.53 && \textbf{83.84}	& 82.97 && \textbf{86.93}	& \textbf{86.93} && \textbf{86.67}	& 85.19 \\
%	700		&& 87.76 			& 85.53 && \textbf{83.84}	& 82.97 && \textbf{86.93}	& \textbf{86.93} && 86.30			& 85.19 \\
%	800		&& \textbf{88.71}	& 85.53 && \textbf{83.84}	& 82.97 && \textbf{86.93}	& \textbf{86.93} && \textbf{86.67}	& 85.19 \\
%	900		&& \textbf{88.71}	& 85.53 && \textbf{83.84}	& 82.97 && \textbf{86.93}	& \textbf{86.93} && \textbf{86.67}	& 85.19 \\
	1000	&& \textbf{88.71}	& 85.53 && \textbf{83.84}	& 82.97 && \textbf{86.93}	& \textbf{86.93} && \textbf{86.67}	& 85.19 \\
	\bottomrule
	\end{tabular}
	\caption{TSA of the affine classifier after $ k $ iterations of OGAProx for the groups according to ``age'', averaged over 5 random partitions.}
	\label{tab:age}
\end{table*}

{\bf \noindent Acknowledgments.} The work of ERC is supported by FWF (Austrian Science Fund), project P 29809-N32.

\end{document}